\documentclass[11pt, reqno]{amsart}
\usepackage[colorlinks,linkcolor=blue,urlcolor=cyan,citecolor=blue,pagebackref]{hyperref}
\usepackage{amssymb}
\usepackage{amsmath,mathrsfs}
\usepackage{mathtools}
\usepackage{mdwtab}
\usepackage{subcaption}
\usepackage{stmaryrd}
\usepackage{array}   % 提供数组表格支持
\usepackage{tikz}    % 用于绘制简单图形
\usepackage{caption} % 标题控制
\usepackage{booktabs} % 提供更好的表格线
\usepackage{longtable}
\usepackage{threeparttable}
\usetikzlibrary{positioning,fit}
\usepackage[indent]{parskip}
\usepackage{url}
\usepackage{fullpage}
\usepackage{pstricks,pst-node}
\usepackage{makecell}
\usepackage{stmaryrd}
\usetikzlibrary{positioning,fit}
\usepackage[indent]{parskip}
\usepackage{url}
\usepackage{enumitem}
\usepackage{fullpage}
\usepackage{pstricks,pst-node}
\usepackage{pifont}
\usepackage{graphicx}
\usepackage{multirow}
\usepackage{fancyhdr} 
\usepackage[enableskew]{youngtab} 
\usepackage{genyoungtabtikz}
\usepackage{float}
\usepackage{ytableau}

\theoremstyle{plain}
\newtheorem{thm}{Theorem}[section]
\newtheorem{prop}[thm]{Proposition}
\newtheorem{lemma}[thm]{Lemma}
\newtheorem{cor}[thm]{Corollary}
\newtheorem{thm:intro}[thm]{thm:introecture}
\newtheorem*{thm*}{Theorem}

\theoremstyle{definition}
\newtheorem{dfn}[thm]{Definition}
\newtheorem{ex}[thm]{Example}

\newtheorem{rem}[thm]{Remark}

\numberwithin{equation}{section}

\DeclareMathAlphabet\mathsf{OT1}{cmss}{m}{n}

\newcommand{\fb}{{\mathfrak b}}

\newcommand{\fg}{{\mathfrak g}}
\newcommand{\fh}{{\mathfrak h}}
\newcommand{\fk}{{\mathfrak k}}

\newcommand{\fm}{{\mathfrak m}}

\newcommand{\fp}{{\mathfrak p}}
\newcommand{\fq}{{\mathfrak q}}

\newcommand{\fu}{{\mathfrak u}}
\newcommand{\fs}{{\mathfrak s}}

\newcommand{\R}{\mathbb R}           % Use for real numbers.
\def\SL{\operatorname{SL}}
\def\SO{\operatorname{SO}}

\newcommand{\rQ}{\operatorname{Q}} 
\newcommand{\rR}{\operatorname{R}}

\newcommand\varIota{\mathrm{I}}

\newcommand{\su}{\mathfrak{su}}

\DeclareMathOperator{\Sp}{Sp}
\renewcommand{\sp}{\mathfrak{sp}}

\newcommand{\so}{\mathfrak{so}}

\newcommand{\g}{\mathfrak{g}}

\newcommand{\la}{\lambda}

\newcommand{\gk}{\operatorname{GKdim}}

\newcommand{\AV}{\mathrm{AV}}

\newcommand\RR{\mathbb R}

\newcommand{\hobox}[3]{\draw (0.0+#1, -1.0-#2) rectangle (1.0+#1, 0.0-#2) node [midway, inner sep=0pt, scale=0.7] {$#3$};}

\renewcommand{\arraystretch}{.75}

\begin{document}

\title[Unitarity of highest weight Harish-Chandra modules]{Unitarity of highest weight Harish-Chandra modules %with trivial infinitesimal character
and smoothness of Schubert varieties}

\author{Zhanqiang Bai}
\address{Department of Mathematical Sciences, Soochow University, Suzhou 215006, China}
\email{\tt  zqbai@suda.edu.cn}

\author{William Q. Erickson}
\address{
Department of Mathematics, Monmouth College, Monmouth, Illinois, USA} 
\email{william.q.erickson@gmail.com}

\author{Markus Hunziker}
\address{
Department of Mathematics, Baylor University, Waco, Texas, USA} 
\email{Markus\_Hunziker@baylor.edu}

\author{Jing Jiang}
\address{
School of Mathematical Sciences, East China Normal University, Shanghai 200241, China} 
\email{jjsd6514@163.com}

\begin{abstract}
Let $G_{\mathbb{R}}$ be a Lie group of Hermitian type, and  $L(\lambda)$ a highest weight Harish-Chandra module of $G_{\mathbb{R}}$ with highest weight $\lambda$.  In this article, we exhibit a bijection between the set of connected Dynkin subdiagrams containing the noncompact simple root and the set of unitary highest weight modules $L(-w\rho-\rho)$, where $\rho$ is half the sum of positive roots. We find that $L(-w\rho-\rho)$ is unitary if and only if the Schubert variety $X(w)$ is smooth.  We also give the cardinality of the set of unitary highest weight modules $L(-w\rho-\rho)$ for each Kazhdan--Lusztig right cell.

\end{abstract}



\maketitle

\tableofcontents
 \setcounter{secnumdepth}{4} 
 \setcounter{tocdepth}{4}

\section{Introduction}\label{intro}

Let \(G_{\mathbb{R}}\) be a connected noncompact simple Lie group with finite center, and let \(K_{\mathbb{R}}\) be a maximal compact subgroup of \(G_{\mathbb{R}}\). Denote by \(\mathfrak{g}_{\mathbb{R}} = \mathrm{Lie}(G_{\mathbb{R}})\) the Lie algebra of \(G_{\mathbb{R}}\). Harish-Chandra showed in \cite[\S1-2]{HC55} that infinite-dimensional highest weight Harish-Chandra modules for \(G_{\mathbb{R}}\) exist precisely when \(G_{\mathbb{R}}\) is of Hermitian type, that is, when \((G_{\mathbb{R}}, K_{\mathbb{R}})\) is a Hermitian symmetric pair. The unitarity of such modules has been extensively studied, with the complete classification  obtained independently by Enright--Howe--Wallach \cite{EHW} and Jakobsen \cite{Ja1}. A more elementary and uniform approach was later given by Enright--Joseph \cite{EJ}, avoiding much of the earlier case-by-case analysis. More recently,  another approach has been developed in \cite{PPST23,PPST25} that revisits the classification of unitary highest weight modules using slightly different methods from those of \cite{EHW} and \cite{Ja1};
meanwhile, a simple uniform characterization of unitarity for highest weight Harish-Chandra modules with prescribed associated variety was established in \cite[Thm. 1.1]{BH2025}, which was expressed directly in terms of the highest weight.
In this paper, we will present a new characterization of unitary highest weight Harish-Chandra modules with trivial infinitesimal character (see Theorem~\ref{thm:bij}).

Throughout, we assume \((G_{\mathbb{R}}, K_{\mathbb{R}})\) is a Hermitian symmetric pair. Let \(G\) and \(K\) denote the complexification of \(G_{\mathbb{R}}\) and \(K_{\mathbb{R}}\), respectively, and let \((\mathfrak{g}, \mathfrak{k})\) denote the complexified Lie algebras of \((G_{\mathbb{R}}, K_{\mathbb{R}})\). As a \(K\)-representation, \(\mathfrak{g}\) decomposes as
\[
\mathfrak{g} = \mathfrak{p}^{-} \oplus \mathfrak{k} \oplus \mathfrak{p}^{+}.
\]
Let \(\mathfrak{h} \subset \mathfrak{k}\) be a Cartan subalgebra, which is also a Cartan subalgebra of \(\mathfrak{g}\). Fix a Borel subalgebra \(\mathfrak{b} \supset \mathfrak{k}\), and let \(\Phi\) and \(\Phi(\mathfrak{k})\) be the root systems of \((\mathfrak{g}, \mathfrak{h})\) and \((\mathfrak{k}, \mathfrak{h})\), respectively. Choose a set of positive roots \(\Phi^{+}(\mathfrak{k})\), and define a positive system for \(\Phi\) by
\[
\Phi^{+} := \Phi^{+}(\mathfrak{k}) \cup \Phi(\mathfrak{p}^{+}).
\]
For a \(\Phi^{+}(\mathfrak{k})\)-dominant integral weight \(\lambda \in \mathfrak{h}^{*}\), let \(F(\lambda)\) be the irreducible \(\mathfrak{k}\)-module with highest weight \(\lambda\). Viewing \(F(\lambda)\) as a \(\mathfrak{q} = \mathfrak{k} \oplus \mathfrak{p}^{+}\)-module (with the nilradical acting trivially), define the induced module
\[
N(\lambda) := U(\mathfrak{g}) \otimes_{U(\mathfrak{q})} F(\lambda).
\]
Let \(L(\lambda)\) denote the irreducible quotient of \(N(\lambda)\).
Each such $L(\la)$ is said to be a \emph{highest weight Harish-Chandra module} for \(\mathfrak{g}\) (or for \(G_{\mathbb{R}}\)). If \(L(\lambda) \neq N(\lambda)\), then we say that \(\lambda\) is a \emph{reduction point}.

Let \(w_c\) and \(w_0\) be the longest elements of the Weyl groups \(W(\mathfrak{k})\) and \(W := W(\Phi)\), respectively. For \(w \in W\), define \(\tilde{w} := w_c w w_0\). Our main result is the following.

\begin{thm}\label{thm:bij}
Let \(G_{\mathbb{R}}\) be a Lie group of Hermitian type, with complexified Lie algebra \(\mathfrak{g}\). Let \(\Gamma\) be the Dynkin diagram of \(\mathfrak{g}\), with \(\alpha\) the noncompact simple root. Then there is a bijection
\[
\left\{
\begin{array}{l}
\textup{connected Dynkin subdiagrams of } \Gamma \\
\textup{containing the noncompact simple root } \alpha
\end{array}
\right\}
 \longleftrightarrow 
\left\{
\begin{array}{l}
\textup{unitary highest weight Harish-Chandra} \\
\textup{modules with infinitesimal character } \rho
\end{array}
\right\},
\]
where the empty subdiagram is included in the left-hand set. Explicitly, a subdiagram \(\Gamma' \subseteq \Gamma\) corresponds to the module \(L(w(\Gamma')\rho - \rho) = L(-w_c x \rho - \rho)\), where \(w(\Gamma') = \tilde{x}\) is defined as in Lemma~\ref{L:involution}. Conversely, if \(L(-w\rho - \rho)\) is unitary, then the corresponding subdiagram has simple roots
\[
\mathrm{supp}(x) := \{\alpha_i \in \Pi \mid s_i = s_{\alpha_i} \textup{ appears in a reduced expression of } x\},
\]
where \(x = w_c w\).
\end{thm}

Let \(B \subset G\) be the Borel subgroup corresponding to \(\mathfrak{b}\), giving a triangular decomposition \(\mathfrak{g} = \mathfrak{n} \oplus \mathfrak{h} \oplus \mathfrak{n}^-\). Fix a parabolic subgroup \(Q\) with \(B \subset Q \subset G\) and \(\mathrm{Lie}(Q_{\mathrm{I}}) = \mathfrak{q} := \mathfrak{k} \oplus \mathfrak{p}^+\). The flag variety \(G/B\) decomposes into Schubert cells \(B w B/B\), indexed by \(w \in W\). Their closures \(X(w) := \overline{B w B/B}\) are Schubert varieties, which are fundamental objects in algebraic geometry, representation theory, and combinatorics. A Schubert variety \(X(w)\) is \emph{smooth} if it is smooth as an algebraic variety, and \emph{rationally smooth} if its Poincaré polynomial \(p_w(t) = \sum_{v \leq w} t^{\ell(v)}\) is palindromic, where \(\leq\) is the Bruhat order and \(\ell(v)\) the length of \(v\).

The study of  smoothness and rational smoothness of Schubert varieties has been a central theme since the foundational work of Lakshmibai and Sandhya \cite{LS-90}, who gave the following combinatorial criterion in type A: \(X(w) \subset \SL(n, \mathbb{C})/B\) is smooth if and only if the permutation \(w \in S_n\) avoids the patterns \(3412\) and \(4231\). Billey \cite{Bi-98} extended this to other classical types, and Billey--Postnikov \cite{BP-05} introduced a uniform approach via root subsystems, showing that (rational) smoothness of \(X(w) \subset G/B\) is equivalent to avoiding a finite list of patterns from stellar root subsystems.

Let \(W^{\mathfrak{k}}\) (resp., \({}^{\mathfrak{k}}W\)) denote the set of minimal length coset representatives of \(W/W(\mathfrak{k})\) (resp., \(W(\mathfrak{k}) \backslash W\)). Kostant \cite{Ko-63} generalized Schubert cell decompositions to \(G/Q\), with cells parametrized by \(W^{\mathfrak{k}}\). When \(G/Q\) is a cominuscule flag variety (i.e., a Hermitian symmetric space), \(W^{\mathfrak{k}}\) admits a description via generalized Young diagrams \cite{Pr-81,EHP}. For classical \(G\) and cominuscule \(G/Q\), Lakshmibai--Weyman \cite{LW-90} described the singular locus of Schubert varieties in terms of partitions. Robles \cite{Ro-14} gave a uniform description of the singular locus of a Schubert
variety $X_Q(w)\subset G/Q$. Hong \cite{Hong:07} established a bijection between connected subdiagrams containing the noncompact simple root and smooth Schubert varieties in \(G/Q\). 

By Brylinski--Kashiwara \cite[Prop.~6.4]{BK81}, the support of the localization of \(L(-w\rho - \rho)\) is the Schubert variety \(X(w) \subset G/B\). Combining this with a result of Richmond--Zainoulline \cite{RZ23}, we obtain the following corollary to Theorem~\ref{thm:bij}.

\begin{cor}\label{uni-bij}
With all notation as in Theorem~\ref{thm:bij}, let \(L(-w\rho - \rho)\) be a  highest weight Harish-Chandra module. Then \(L(-w\rho - \rho)\) is unitary if and only if the Schubert variety \(X(w) \subset G/B\) is smooth, which is equivalent to the smoothness of \(X_Q(x^{-1}) \subset G/Q\), where \(x = w_c w \in {}^{\mathfrak{k}}W\).
\end{cor}

From Barchini--Zierau \cite{BZ-17}, it seems that one direction of Corollary~\ref{uni-bij} was already known, namely, that the support of the localization of  a unitary highest weight Harish-Chandra module \(L(-w\rho - \rho)\) is smooth. 
The converse, however, seems to be new.
To determine the unitarity of \(L(-w\rho - \rho)\), we have implemented an algorithm in Python, available at
\begin{center}
\url{https://github.com/JingJiang-web/unitary-for-Hermitian-type}.
\end{center}
% \[
% \href{https://github.com/JingJiang-web/unitary-for-Hermitian-type}{\texttt{https://github.com/JingJiang-web/unitary-for-Hermitian-type}}.
% \]

When \(\Gamma'\) is a connected subdiagram of \(\Gamma\), the module \(L(w(\Gamma')\rho - \rho)\) is called a \emph{special Kostant module} in \cite{BH:09}; we refer to such modules in Theorem~\ref{thm:bij} as \emph{special Kostant modules of Hermitian type}. Kostant modules are those simple highest weight modules admitting a Bernstein--Gelfand--Gelfand (BGG) resolution \cite{BH:09}. Boe--Hunziker \cite{BH:09} gave a classification of Kostant modules in regular blocks for maximal parabolics in the simply laced types. For Hermitian symmetric pairs, they found a uniform characterization of  Kostant modules in regular blocks via subdiagrams of some Dynkin diagram containing the noncompact simple root. Kostant modules are significant, among other reasons, because every unitary highest weight module is a Kostant module; see Theorem~\ref{kostant-module} below, from Enright~\cite{Enright}. Thus one can use certain properties of Kostant modules to study unitary highest weight modules; see, for example, \cite{EH-04,EW,EH-dim}.

% Since every unitary highest weight module is a Kostant module \cite{Enright}, Theorem~\ref{thm:bij} implies a bijection between smooth Schubert varieties in \(G/Q\) and unitary highest weight Harish–Chandra modules with trivial infinitesimal character.

% In simply laced types, rational smoothness and smoothness coincide. Boe--Hunziker \cite{BH:09} showed that Kostant highest weight modules correspond bijectively to rationally smooth Schubert varieties in \(G/Q\). Hence, for simply laced types, there is a one-to-one correspondence between:
% \begin{itemize}
% \item connected subdiagrams of \(\Gamma\) containing \(\alpha\),
% \item special Kostant modules of Hermitian type (unitary highest weight Harish–Chandra modules), and
% \item smooth Schubert varieties in \(G/Q\).
% \end{itemize}
% In non-simply laced types, however, there exist Kostant highest weight modules that are not unitary, corresponding to Schubert varieties that are rationally smooth but not smooth.

% The paper is organized as follows. Section~\ref{pre} collects necessary preliminaries. Section~\ref{uniform-proof} gives a uniform proof of Theorem~\ref{thm:bij} for simply laced types. Section~\ref{case-proof} provides a case-by-case proof for non-simply laced types using Boe–Hunziker’s characterization of Kostant modules \cite{BH:09}. Finally, Section~\ref{number} uses the Gelfand–Kirillov dimension formula for unitary highest weight modules from Bai--Hunziker \cite{BH} to count such modules in each Kazhdan–Lusztig right cell. We follow the notation and coordinates of \cite{EHW} throughout.

The paper is organized as follows. In \S \ref{pre}, we  prepare some necessary preliminaries. In \S \ref{uniform-proof}, we give a uniform proof of Theorem \ref{thm:bij} for simply laced types. In \S \ref{case-proof}, we give a case-by-case proof of Theorem \ref{thm:bij} for non-simply laced types by using the characterization of Kostant modules given by Boe--Hunziker \cite{BH:09}. In \S \ref{number}, by using the formula of Gelfand--Kirillov dimensions of unitary highest weight modules in Bai--Hunziker \cite{BH}, we count such unitary highest
weight modules with the trivial infinitesimal character for each Kazhdan--Lusztig right cell. Throughout, we follow the notation in~\cite{EHW} in writing coordinates for root
systems and other relevant data.

\section{Preliminaries}\label{pre}

% \subsection{Parabolic subalgebras of Hermitian type}

Let $\fg$   be a complex simple Lie algebra of rank $n$. We fix a Cartan subalgebra $\fh$ of $\fg$, and denote 
by $\Phi\subset \fh^*$ the root system of $\fg$ with respect to $\fh$.
For every $\alpha\in \Phi$, we write $\fg_\alpha$ to denote the one-dimensional root space. 
We choose a simple system  $\Pi\subset \Phi$ and write $\Pi=\{\alpha_1,\ldots,\alpha_n\}$, where the simple 
roots $\alpha_i$ are labeled according to Bourbaki~\cite{Bour}. This simple system determines a positive system $\Phi^+\subset \Phi$.
The corresponding standard Borel subalgebra is defined as
$$
    \fb:=\fh\oplus  \bigoplus_{\alpha\in \Phi^+} \fg_\alpha.
$$
 Every Borel subalgebra of $\fg$ is conjugate to $\fb$ under 
the action of the adjoint group of $\fg$. 

\subsection{Parabolic subalgebras}
A \emph{parabolic subalgebra} $\fq$ of $\fg$ is a subalgebra containing some Borel subalgebra of $\fg$.
If $\fq$ contains the fixed Borel subalgebra $\fb$, we say that $\fq$ is a \emph{standard} 
parabolic subalgebra.  Any parabolic subalgebra $\fq$ of $\fg$ 
is conjugate to a standard parabolic subalgebra under 
the action of the adjoint group of $\fg$. 
There is a one-to-one  correspondence between subsets of
$\Pi$  and standard parabolic subalgebras of $\fg$ as follows.
For $\varIota\subset \Pi$, define a root system $\Phi_{\varIota}$ by 
$$
    \Phi_{\varIota} :=\Phi\cap\sum_{\alpha\in \varIota} \mathbb{Z}\alpha.
$$
Next define a reductive subalgebra $\fm_{\varIota}$  of $\fg$ and a nilpotent subalgebra  $\fu_{\varIota}$ 
of $\fg$ by 
$$
\fm_{\varIota} :=\fh\oplus  \bigoplus_{\alpha\in \Phi_{\varIota}} \fg_\alpha\quad \mbox{and} 
\quad \fu_{\varIota} :=\bigoplus_{\alpha\in \Phi^+\setminus \Phi_{\varIota}} \fg_\alpha .
$$
Then $\fq_{\varIota}:=\fm_{\varIota}\oplus\fu_{\varIota}$ is a standard parabolic subalgebra of $\fg$ 
with Levi subalgebra $\fm_{\varIota}$ and nilradical $\fu_{\varIota}$. 
If $|\Pi\setminus\varIota|=1$, then $\fq_{\varIota}$ is called a maximal parabolic subalgebra. 

To simplify notation, we often omit the subscript and write, for instance, $\fq=\fm\oplus\fu$ instead of 
$\fq_{\varIota}=\fm_{\varIota}\oplus\fu_{\varIota}$. The set 
$$
\Phi(\fu):=\Phi^+\setminus \Phi_\varIota
$$
is a partially ordered subset of $\Phi^+$.

Now suppose we have an irreducible Hermitian symmetric pair $(G_{\mathbb{R}}, K_{\mathbb{R}})$ of noncompact type.  Let $(\fg_\RR,\fk_\RR)$ be
the pair of Lie algebras of  $(G_\RR, K_\RR)$. Then $\fg_\RR$ is a noncompact real form of $\fg$ with Cartan decomposition $\fg_\RR=\fk_\RR\oplus \fp_\RR$,
and we may assume that $\fk_\RR$ has a nontrivial center $\mathfrak{z}_\RR$. We have the usual decomposition $\mathfrak{g}=\mathfrak{p}^{-}\oplus \mathfrak{k}\oplus \mathfrak{p}^{+}$ as a $K$-representation. We say that $\mathfrak{q} := \mathfrak{k}+\mathfrak{p}^+$ is a parabolic subalgebra of \emph{Hermitian type}.

If $\fq=\fm\oplus\fu=\mathfrak{k}\oplus\mathfrak{p}^+$ is of Hermitian type, then the roots in $\Phi_{\varIota}$ are called the \emph{compact roots}
and the roots in $\Phi\setminus \Phi_{\varIota}$ are called the \emph{noncompact roots}. Note that there
is  a unique noncompact simple root, namely the root $\alpha\in \Pi\setminus\varIota$.
Usually we denote the set of positive compact roots by $\Phi^+(\mathfrak{k}):=\Phi_{\varIota}^+$ and the set of positive noncompact roots by $\Phi(\mathfrak{p}^+):=\Phi(\fu)=\Phi^+\setminus \Phi_{\varIota}$. See Table \ref{hermitian} for the classification of Hermitian symmetric pairs.

\vspace{.5pc}
\begin{table}[h]

\centering
\begingroup
%{{\bf Table 1}\quad Hermitian symmetric pairs of noncompact type} \\
%\vspace{.5pc}
\renewcommand{\arraystretch}{1.5}
\setlength\tabcolsep{15pt}
\begin{tabular}{l|l|l}
 \hline
$\fg_{\mathbb{R}}$ & $\fk_\RR$& Dynkin diagram\\
\hline
% \noalign{\smallskip}
%  & & \\[-7pt]
$\su(p,q)$ & $\fs(\fu(p)\oplus \fu(q))$ & 
\begin{pspicture}(-.10,0.3)(2.6,.1)
\cnode*(0,0){.07}{d1}
\uput[d](0,0){\small $1$}
\cnode*(0.5,0){.07}{d3}
\uput[d](0.5,0){\small $2$}
\cnode*(1 ,0){.07}{d4}
\uput[d](1,0){\small $\cdots$}
\cnode*(1.5,0){.07}{d5}
\uput[d](1.5,-.065){\small $p$}
\cnode*(2,0){.07}{d6}
\uput[d](2,0){\small $\cdots$}
\cnode*(2.5,0){.07}{d7}
\uput[d](2.5,-.065){\small $n$}
\ncline{d1}{d3}
\ncline[linestyle=dotted,dotsep=1.3pt]{d3}{d4}
\ncline{d4}{d5}
\ncline{d5}{d6}
\ncline[linestyle=dotted,dotsep=1.3pt]{d6}{d7}
\pscircle[linewidth=.5pt,fillstyle=solid](0,0){.07}
\pscircle[linewidth=.5pt,fillstyle=solid](0.5,0){.07}
\pscircle[linewidth=.5pt,fillstyle=solid](1.0,0){.07}
\pscircle[linewidth=.5pt,fillstyle=solid](2,0){.07}
\pscircle[linewidth=.5pt,fillstyle=solid](2.5,0){.07}
\end{pspicture}
\\
 & & \\ \hline
$\so(2n-1,2)$ & $\so(2n-1)\oplus \so(2)$ & 
 \begin{pspicture}(-.10,0.3)(2.6,.1)
\cnode*(0,0){.07}{d1}
\uput[d](0,0){\small $1$}
\cnode*(0.5,0){.07}{d3}
\uput[d](0.5,0){\small $2$}
\cnode*(1 ,0){.07}{d4}
\cnode*(1.5,0){.07}{d5}
\uput[d](1.5,-.065){\small $\cdots$}
\cnode*(2,0){.07}{d6}
\cnode*(2.5,0){.07}{d7}
\uput[d](2.5,-.065){\small $n$}
\ncline{d1}{d3}
\ncline[linestyle=dotted,dotsep=1.3pt]{d3}{d4}
\ncline{d4}{d5}
\ncline{d5}{d6}
\psline[linewidth=.5pt](2,.06)(2.5,.06)
\psline[linewidth=.5pt](2,-.06)(2.5,-.06)
\uput[r](1.925,0){$>$}
\pscircle[linewidth=.5pt,fillstyle=solid](0.5,0){.07}
\pscircle[linewidth=.5pt,fillstyle=solid](1,0){.07}
\pscircle[linewidth=.5pt,fillstyle=solid](1.5,0){.07}
\pscircle[linewidth=.5pt,fillstyle=solid](2,0){.07}
\pscircle[linewidth=.5pt,fillstyle=solid](2.5,0){.07}
\end{pspicture}
 \\ 
  & & \\ \hline
$\sp(n, \RR)$ & $\fu(n)$ &
 \begin{pspicture}(-.10,0.3)(2.6,.1)
\cnode*(0,0){.07}{d1}
\uput[d](0,0){\small $1$}
\uput[d](0.5,0){\small $2$}
\uput[d](1.5,-.065){\small $\cdots$}
\uput[d](2.5,-.065){\small $n$}
\cnode*(0.5,0){.07}{d3}
\cnode*(1 ,0){.07}{d4}
\cnode*(1.5,0){.07}{d5}
\cnode*(2,0){.07}{d6}
\cnode*(2.5,0){.07}{d7}
\ncline{d1}{d3}
\ncline[linestyle=dotted,dotsep=1.3pt]{d3}{d4}
\ncline{d4}{d5}
\ncline{d5}{d6}
\psline[linewidth=.5pt](2,.06)(2.5,.06)
\psline[linewidth=.5pt](2,-.06)(2.5,-.06)
\uput[r](1.925,0){$<$}
\pscircle[linewidth=.5pt,fillstyle=solid](0,0){.07}
\pscircle[linewidth=.5pt,fillstyle=solid](0.5,0){.07}
\pscircle[linewidth=.5pt,fillstyle=solid](1,0){.07}
\pscircle[linewidth=.5pt,fillstyle=solid](1.5,0){.07}
\pscircle[linewidth=.5pt,fillstyle=solid](2,0){.07}
\end{pspicture} \\ 
  & & \\ \hline
$\so(2n-2,2)$ & $\so(2n-2)\oplus \so(2)$ & 
 \begin{pspicture}(-.10,0.8)(2.6,.1)
\cnode*(0,0){.07}{d1}
\uput[d](0,0){\small $1$}
\cnode*(0.5,0){.07}{d2}
\uput[d](0.5,0){\small $2$}
\uput[d](1.5,-.065){\small $\cdots$}
\cnode*(1 ,0){.07}{d3}
\cnode*(1.5,0){.07}{d4}
\cnode*(2,0){.07}{d5}
\cnode*(2.354,0.354){.07}{d6}
\uput[r](2.354,0.354){\small $n-1$}
\cnode*(2.354,-0.354){.07}{d7}
\uput[r](2.354,-0.354){\small $n$}
\ncline{d1}{d2}
\ncline[linestyle=dotted,dotsep=1.3pt]{d2}{d3}
\ncline{d3}{d4}
\ncline{d4}{d5}
\ncline{d5}{d6}
\ncline{d5}{d7}
\pscircle[linewidth=.5pt,fillstyle=solid](0.5,0){.07}
\pscircle[linewidth=.5pt,fillstyle=solid](1,0){.07}
\pscircle[linewidth=.5pt,fillstyle=solid](1.5,0){.07}
\pscircle[linewidth=.5pt,fillstyle=solid](2,0){.07}
\pscircle[linewidth=.5pt,fillstyle=solid](2.354,.354){.07}
\pscircle[linewidth=.5pt,fillstyle=solid](2.354,-.354){.07}
\end{pspicture}\\ 
  & & \\ \hline
$\so^*(2n)$ & $\fu(n)$& 
 \begin{pspicture}(-.10,0.8)(2.6,.1)
\cnode*(0,0){.07}{d1}
\uput[d](0,0){\small $1$}
\cnode*(0.5,0){.07}{d2}
\uput[d](0.5,0){\small $2$}
\cnode*(1 ,0){.07}{d3}
\uput[d](1.5,-.065){\small $\cdots$}
\cnode*(1.5,0){.07}{d4}
\cnode*(2,0){.07}{d5}
\cnode*(2.354,0.354){.07}{d6}
\uput[r](2.354,0.354){\small $n-1$}
\cnode*(2.354,-0.354){.07}{d7}
\uput[r](2.354,-0.354){\small $n$}
\ncline{d1}{d2}
\ncline[linestyle=dotted,dotsep=1.3pt]{d2}{d3}
\ncline{d3}{d4}
\ncline{d4}{d5}
\ncline{d5}{d6}
\ncline{d5}{d7}
\pscircle[linewidth=.5pt,fillstyle=solid](0,0){.07}
\pscircle[linewidth=.5pt,fillstyle=solid](0.5,0){.07}
\pscircle[linewidth=.5pt,fillstyle=solid](1,0){.07}
\pscircle[linewidth=.5pt,fillstyle=solid](1.5,0){.07}
\pscircle[linewidth=.5pt,fillstyle=solid](2,0){.07}
\pscircle[linewidth=.5pt,fillstyle=solid](2.354,.354){.07}
\end{pspicture}\\ 
  & & \\ \hline
%${\rm E_{III}}$ 
$\mathfrak{e}_{6(-14)}$
& $\so(10)\oplus \RR$ & 
\begin{pspicture}(-.10,0.8)(2.6,.1)
\cnode*(0,0){.07}{d1}
\uput[d](0,0){\small $1$}
\cnode*(1,.5){.07}{d2}
\uput[r](1,.5){\small $2$}
\cnode*(0.5,0){.07}{d3}
\uput[d](0.5,0){\small $3$}
\cnode*(1 ,0){.07}{d4}
\uput[d](1,0){\small $4$}
\cnode*(1.5,0){.07}{d5}
\uput[d](1.5,0){\small $5$}
\cnode*(2,0){.07}{d6}
\uput[d](2,0){\small $6$}
\ncline{d1}{d3}
\ncline{d3}{d4}
\ncline{d2}{d4}
\ncline{d4}{d5}
\ncline{d5}{d6}
\pscircle[linewidth=.5pt](0,0){.07}
\pscircle[linewidth=.5pt,fillstyle=solid](0.5,0){.07}
\pscircle[linewidth=.5pt,fillstyle=solid](1,0){.07}
\pscircle[linewidth=.5pt,fillstyle=solid](1,.5){.07}
\pscircle[linewidth=.5pt,fillstyle=solid](1.5,0){.07}
\pscircle[linewidth=.5pt,fillstyle=solid](2,0){.07}
\end{pspicture}\\ 
  & & \\ \hline
%${\rm E_{VII}}$ 
$\mathfrak{e}_{7(-25)}$
& $\mathfrak{e}_6\oplus \RR$ & 
\begin{pspicture}(-.10,0.8)(2.6,.1)
\cnode*(0,0){.07}{d1}
\uput[d](0,0){\small $1$}
\cnode*(1,.5){.07}{d2}
\uput[r](1,.5){\small $2$}
\cnode*(0.5,0){.07}{d3}
\uput[d](0.5,0){\small $3$}
\cnode*(1 ,0){.07}{d4}
\uput[d](1,0){\small $4$}
\cnode*(1.5,0){.07}{d5}
\uput[d](1.5,0){\small $5$}
\cnode*(2,0){.07}{d6}
\uput[d](2,0){\small $6$}
\cnode*(2.5,0){.07}{d7}
\uput[d](2.5,0){\small $7$}
\ncline{d1}{d3}
\ncline{d3}{d4}
\ncline{d2}{d4}
\ncline{d4}{d5}
\ncline{d5}{d6}
\ncline{d6}{d7}
\pscircle[linewidth=.5pt,fillstyle=solid](0,0){.07}
\pscircle[linewidth=.5pt,fillstyle=solid](0.5,0){.07}
\pscircle[linewidth=.5pt,fillstyle=solid](1,0){.07}
\pscircle[linewidth=.5pt,fillstyle=solid](1,.5){.07}
\pscircle[linewidth=.5pt,fillstyle=solid](1.5,0){.07}
\pscircle[linewidth=.5pt,fillstyle=solid](2,0){.07}
\end{pspicture}\\
 & & \\ \hline
\end{tabular}
\caption{Hermitian symmetric pairs of noncompact type}\label{hermitian}
\endgroup
%\vspace{1pc}
\end{table}

  \subsection{Minimal length coset representatives}

In the following, let $(\ ,\ )$ denote the  nondegenerate bilinear form on $\fh^*$ induced from the Killing form of $\fg$.
For $\alpha\in \Phi$, set $\alpha^\vee:=2\alpha/(\alpha,\alpha)$ and define
the reflection $s_\alpha: \fh^*\rightarrow \fh^*$ by 
$s_{\alpha}(\lambda) := \lambda-(\lambda,\alpha^\vee)\alpha$. The  Weyl group $W$
 is generated by the simple reflections $s_\alpha$, for all $\alpha\in \Pi$.
The \emph{length} of an element $w\in W$, denoted by $\ell(w)$, is the length of the shortest word representing $w$ 
as a product of the simple reflections (called a \emph{reduced expression} of $w$).
By \cite[\S 10.3]{hum72}, $\ell(w)$ is also equal to the number of positive roots $\alpha$ for which $w(\alpha)\in \Phi^-$.
Therefore, upon defining
\[
    \Phi_w:=\Phi^+\cap w\Phi^-,%=\{\alpha \in \varPhi^+\mid w^{-1}\alpha  \in \varPhi^-\}.
\]
%\end{defn}
we have $|\Phi_w|=\ell(w)$. For any subset $\Psi\subset \Phi^+$, we define $\langle \Psi \rangle:=\sum_{\alpha \in \Psi} \alpha$.

\begin{lemma}[{\cite[(5.10.1)]{Kostant:61}}]\label{L:rho}
For every $w\in W$,
$$
    w\rho = \rho-\langle \Phi_w\rangle,
$$
where $\rho:=\frac{1}{2}\langle \Phi^+\rangle$ as usual.
\end{lemma}

Let $\fq=\fq_{\varIota}=\mathfrak{k}\oplus\mathfrak{p}^+$ be a  parabolic subalgebra of Hermitian type, and let $ W(\mathfrak{k}):=W_{\varIota}=\langle s_\alpha \mid \alpha \in \varIota\rangle$ be  the Weyl group of the Levi subalgebra  $\fk$.

\begin{dfn}\label{D:IW}
Following Kostant~\cite[(5.13.1)]{Kostant:61}, define
$$
{}^\fk W:=\{w\in W \mid \Phi_w \subset \Phi(\mathfrak{p}^+)\}.
$$
For later reference, we also define $W^\fk:=\{w^{-1} \mid w\in {}^\fk W\}$.
\end{dfn}

% When $\fq=\fm\oplus\fu=\mathfrak{k}\oplus\mathfrak{p}^+$ is of Hermitian type, we usually denote ${}^\mathfrak{k} W:={}^\varIota W$ and $W(\mathfrak{k}):=W_{\varIota}$.

\begin{prop}[{\cite[Prop.~5.13]{Kostant:61}}]\label{decomp}
Every element $w\in W$ can be  uniquely written in the form 
$ w=uv$, where $u\in W(\mathfrak{k})$ and $v\in {}^\fk W$. Furthermore, if $w=uv$
is such a decomposition, then  $\ell(w)=\ell(u)+\ell(v)$.
Similarly,  $w$ can be  uniquely written in the form 
$ w=vu$, where $u\in W(\mathfrak{k})$ and $v\in {W}^{\fk}$.
\end{prop}

\begin{cor}[{\cite[Cor.~3.5 and Rem.~3.6]{EHP}}]
The set ${}^\fk W$ is the set of minimal length coset representatives of $W(\fk)\backslash W$
(i.e., the set of right cosets of $W(\fk)$). For $w\in W$, the following are equivalent:
\begin{enumerate}
\item[\rm{(i)}] $\Phi_w\subset \Phi(\mathfrak{p}^+)$.
\item[\rm{(ii)}] $w^{-1}\Phi^+(\fk) \subset \Phi^+$.
\item[\rm{(iii)}] \normalfont $w\rho$  is   $\Phi^+(\fk)$-dominant integral.
\item[\rm{(iv)}] $\ell(s_\alpha w)=\ell(w)+1$ for all $\alpha\in \varIota$.
\end{enumerate}
Similarly, the set ${W}^{\fk} $ is the set of minimal length coset representatives of $W/W(\mathfrak{k})$.
\end{cor}

    Let $(\mathcal{P},\leq)$  be a partially ordered set. We say that a subset $\mathcal{I}\subset \mathcal{P}$ is a \emph{lower order ideal} of $\mathcal{P}$  if
$x\in \mathcal{I }$ and $x\geq y\in \mathcal{P}$ implies $y\in \mathcal{I}$. An upper order ideal of  $\mathcal{P}$ is defined analogously.
In the following, we will view the sets $\Phi^+$ and $\Phi(\fu)$ as partially ordered sets with  the usual ordering induced from the ordering on $\fh^*$, 
i.e., $\mu\leq \lambda$ if and only if $\lambda-\mu$ is a (possibly empty) sum of positive roots.
 
 For $v,w\in W$, write $v\rightarrow w$ if $\ell(w)>\ell(v)$ and $w=vs_\alpha$ for some $\alpha\in \Phi^+$.
Then define $v<w$ if there exists a sequence $v=w_1\rightarrow w_2\rightarrow\cdots \rightarrow w_m=w$.
The resulting partial order ``$\leq$'' on $W$ is called the {\it Bruhat order\/}. For a detailed treatment of the Bruhat order, see \cite[Chap. 2]{BB05}.

\begin{lemma}[{\cite[Lemma~3.7, Prop.~3.8 and Cor.~3.12]{EHP}}]\label{L:ideal}
Suppose $\fq$ is a parabolic subalgebra of Hermitian type. Let  $v, w\in {}^\mathfrak{k} W$. Then we have a  bijection
\begin{align*}
    {}^\mathfrak{k} W
    &\rightarrow \{\textup{lower order ideals  in $\Phi(\mathfrak{p}^+)$}\},\\
    w &\mapsto \Phi_w.
\end{align*}
Furthermore, $\Phi(\mathfrak{p}^+) \setminus \Phi_w$
 is an upper order ideal of $\Phi^+$, and we have $v\leq w$ if and only if $\Phi_v\subseteq \Phi_w$.

\end{lemma}

\begin{lemma}[{\cite[Lemma 4.1]{Ja1}}]
\label{lemma:Ja}
Let $\alpha \in \Phi(\mathfrak{p^{+}})$, let $\pi_{1},\ldots,\pi_{k}$
be distinct elements of $\Pi\cap \Phi(\mathfrak{k})$, and assume that $\alpha+\pi_{i}\in \Phi(\mathfrak{p^{+}})$ for $i=1,\ldots, k$. Then $k\leq 2$. Furthermore, if $k=2$, then 
$\pi_{1}\perp\pi_{2}$ and  $\alpha+\pi_{1}+\pi_{2}\in \Phi(\mathfrak{p^{+}})$. \hfill$\Box$
\end{lemma}

In light of Lemma~\ref{lemma:Ja}, the Hasse diagram of $\Phi(\mathfrak{p^{+}})$ is an upward planar graph 
of order dimension 2, and hence (upon rotating) can be drawn on a two-dimensional orthogonal lattice. 

\begin{ex}
Let
$\mathfrak{g}_{\mathbb{R}}=\su(3,2)$. 
Then we have the following (where dots denote zero entries):
\begin{center}
\begin{pspicture}(-7,0)(3,3)
\uput[l](-2,1.3){$
\mathfrak{p}^+=\left\{\left(\begin{array}{ccc|cc} 
\cdot&\cdot&\cdot&\, * & * \\[-4 pt]
\cdot&\cdot&\cdot&\, * & * \\[-4 pt]
\cdot&\cdot&\cdot&\, * & * \\[-4 pt]
\hline
\cdot&\cdot&\, \cdot & \cdot & \cdot\\[-4 pt]
\cdot&\cdot&\, \cdot & \cdot & \cdot
\end{array}
\right)\right\},  
$}
\uput[l](.3,1.3){$\Delta(\mathfrak{p}^+)=$}
\cnode*(1.5,.5){.07}{a0}
\cnode*(1,1){.07}{a1}
\cnode*(2,1){.07}{a2}
\cnode*(.5,1.5){.07}{a3}
\cnode*(1.5,1.5){.07}{a4}
%\cnode*(0,1.5){.07}{a6}
\cnode*(1,2){.07}{a7}
%\cnode*(2,1.5){.07}{a8}
%\cnode*(.5,2){.07}{a9}
%\cnode*(1.5,2){.07}{a10}
%\cnode*(1,2.5){.07}{a11}
\ncline{->}{a0}{a1}
\ncline{->}{a0}{a2}
\ncline{->}{a1}{a3}
\ncline{->}{a1}{a4}
\ncline{->}{a2}{a4}
\ncline{->}{a2}{a5}
\ncline{->}{a3}{a6}
\ncline{->}{a3}{a7}
\ncline{->}{a4}{a7}
\ncline{->}{a4}{a8}
\ncline{->}{a5}{a8}
\ncline{->}{a6}{a9}
\ncline{->}{a7}{a9}
\ncline{->}{a7}{a10}
\ncline{->}{a8}{a10}
\ncline{->}{a9}{a11}
\ncline{->}{a10}{a11}
\uput[r](1.15,.2){$\alpha_{3}=e_3-e_4$}
\uput[r](.75,2.35){$\beta =e_1-e_5$}
\uput[d](1.1,.9){\scriptsize{2}}
\uput[d](0.6,1.4){\scriptsize{1}}
%\uput[d](0.1,1.4){\scriptsize{1}}
\uput[u](0.6,1.6){\scriptsize{4}}
\end{pspicture}
\end{center}
\end{ex}

In \cite[\S 3.6]{EHP}, Enright--Hunziker--Pruett  associated  a generalized Young diagram to each lower order ideal $\Phi_w\subset \Phi(\mathfrak{p}^+)$  as follows.
Start with the poset $\Phi_w$ viewed as a subdiagram of the Hasse diagram of $\Phi(\mathfrak{p}^+)$.
Then replace the nodes of the diagram of $\Phi_w$ with boxes (squares) so that two boxes share a side if and only if the two corresponding nodes are connected.
Finally, rotate the resulting diagram $135^\circ$ clockwise so that the node corresponding to the noncompact simple root  $\alpha\in \Pi\setminus\varIota$
now corresponds to the upper-left box.

% \begin{prop}\label{P: abelian ideals}
% Suppose that $\fq$ is of Hermitian type.  Then we have a  bijection
% $$
%     {}^\varIota W \rightarrow \{\mbox{ideals of  $\fb$ contained in $\fu$}\}
% $$
% given by
% $w\mapsto \fa_w:=\sum_{\alpha\in \varPhi(\fu)\setminus \varPhi_w} \fg_\alpha.$
% \end{prop}

There is a canonical, order-inverting, length complementary involution on ${}^\fk W$, given in the following lemma. 

\begin{lemma}[{\cite[Lemma 3.14 and Cor.~3.15]{EHP} }]\label{L:involution}
Let $w_0$ and $w_{c}$ be the longest elements in $W$ and $W(\fk)$, respectively.
Then the mapping $\ \widetilde{\ }: {}^\fk W\rightarrow  {}^\fk W$
given by $x\mapsto \tilde{x}:=w_{c} xw_0$ is an order-inverting, length-complementary involution.
Moreover, we have $\Phi(\fu)\setminus \Phi_x=w_{c} \Phi_{\tilde{x}}$.
\end{lemma}

% \begin{cor} 
% Also we have $\Phi(\fu)\setminus \Phi_w=w_{\,\varIota} \Phi_{\tilde{w}}$.
% \end{cor}

\begin{dfn}\label{D:labels}
Suppose $\fq$ is of Hermitian type. Define a map 
$$
f: \Phi(\fp^+) \rightarrow \Pi
$$
as follows. For $\beta\in \Phi(\fp^+)$, the sets $\Phi(\fp^+)_{<\beta}:=\{\gamma\in \Phi(\fp^+) \mid \gamma< \beta\}$ and $\Phi(\fp^+)_{\leq \beta}:=\{\gamma\in \Phi(\fp^+) \mid \gamma\leq  \beta\}$ are lower order ideals of $\Phi(\fp^+)$. Therefore, by Lemma~\ref{L:ideal},  there exist $v,w\in  {}^\mathfrak{k} W$ such that 
$\Phi(\mathfrak{p}^+)_{<\beta}=\Phi_v$ and $\Phi(\mathfrak{p}^+)_{\leq \beta}=\Phi_w$. Since $\Phi_w=\Phi_v \dot\cup\{\beta\}$, there exists a unique 
$f(\beta)\in \Pi$ such that $w=vs_{f(\beta)}$, namely $f(\beta):=v^{-1}\beta$. 
\end{dfn}

\begin{prop}[{\cite[Prop.~3.21]{EHP}}]\label{P:product}
Suppose $\fq$ is of Hermitian type.  Let $w\in {}^\mathfrak{k} W$, and write
$$\Phi_w=\{\beta_1,\beta_2, \ldots,\beta_\ell\}$$ such that for every 
$1\leq j \leq \ell$, the set $\{\beta_1,\beta_2, \ldots,\beta_j\}$ is a lower order ideal of $\Phi(\fp^+)$. 
Then
$$
    w= s_{f(\beta_1)}s_{f(\beta_2)}\cdots s_{f(\beta_l)}
$$
is a  reduced expression for $w$. Furthermore, $w=s_{\beta_\ell} s_{\beta_{\ell-1}}\cdots s_{\beta_1}$.
\end{prop}

Proposition~\ref{P:product} gives us a very simple way to write 
canonical reduced expressions for the elements of ${}^\mathfrak{k} W$.
First, fill the boxes of the generalized Young diagram of the longest element of ${}^\mathfrak{k} W$ with the numbers $1,2,\ldots,n$,
so that the box corresponding to $\beta\in \Phi(\fp^+)$  is assigned the number $i$ such that $f(\beta)=\alpha_i$.
The upper-left box is always assigned the index of the unique noncompact simple root in $\Pi\setminus\varIota$. Full details can be found in~\cite[\S3.8]{EHP}.
    
    \subsection{Kostant modules}
Let $\mathfrak{g}$ be a complex simple Lie algebra and $\mathfrak{q}=\mathfrak{q}_{\varIota}$ a parabolic subalgebra of Hermitian type. We assume that $(\Phi,\Phi_{c}):=(\Phi,\Phi_{\varIota})$ corresponds to the Hermitian symmetric pair. Let $\mu\in \mathfrak{h}^*$ be regular, which means that $(\mu+\rho, \alpha^{\vee})\neq 0$ for all $\alpha\in \Phi^+$. The weight $\mu$ is said to be 
 \emph{antidominant} if $(\mu+\rho,\alpha^{\vee})\notin \mathbb{Z}_{> 0}$ for all $\alpha \in \Phi^+$. 
 We fix a regular infinitesimal block $\mathcal{O}_{\mu}^{\mathfrak{q}}$ of the category $\mathcal{O}^{\mathfrak{q}}$. The simple modules are denoted by $L_w:=L(w_{c}w(\mu+\rho)-\rho)$ for $w \in {}^\mathfrak{k} W$. Usually we may choose $\mu=-2\rho$ which is also regular. Thus the simple modules are $L_w=L(-w_{c}w\rho-\rho)$ in the category $\mathcal{O}_{-2\rho}^{\mathfrak{q}}$. 

By \cite{Zie18}, $L(-w\rho-\rho)$ is a highest weight Harish-Chandra module if and only if $-w\rho$ is $\Phi^+(\mathfrak{k})$-dominant. From \cite[Prop.~2.3]{BN05},  $L(-w\rho-\rho)$ is a highest weight Harish-Chandra module if and only if $w=w_cv$ for some $v\in {}^{\fk}W$.

\begin{dfn}
A simple module $L_w$ in $\mathcal{O}_{-2\rho}^{\mathfrak{q}}$ is called a \textit{Kostant module} if, as a $\fk$-module, 
\[
H^i(\mathfrak{u}, L_w) \cong \bigoplus_{\substack{x \leq w, \\ \ell(x,w) = i}} F_x
\]
for $i\geq 0$,
where $F_x$ is the finite-dimensional $\mathfrak{k}$-module with highest weight $-w_{c}x\rho-\rho $, and $\ell(x,w) = \ell(w) - \ell(x)$.
\end{dfn}

 \begin{thm}[{\cite[Thm.~6.2]{BH:09}}]\label{sub-Kostant}
 Let $\mathcal{O}^{\mathfrak{q}}_{-2\rho}$ be the regular  block for a Hermitian symmetric pair of 
 simply laced type, and let $\varIota := \Pi\setminus \{\alpha\}$, where $\alpha$ is the noncompact simple root. Let $\Gamma$ be the corresponding Dynkin diagram.  Then there is a
 bijection
$$
\left\{
\begin{array}{c}
\textup{connected subdiagrams}\\
\textup{of } \Gamma \textup{ containing } \alpha\\  
\end{array}
\right\}
\longleftrightarrow 
\{\textup{Kostant modules in }\mathcal{O}^{\mathfrak{q}}_{-2\rho}\},
$$
where we declare the empty subdiagram to belong to the left-hand set.
 \end{thm}	
When $\mathfrak{g}$ is not simply laced, that is, when $(\Phi,\Phi_{c})=(\mathsf{B}_n,\mathsf{B}_{n-1})$ or $(\mathsf{C}_n,\mathsf{A}_{n-1})$, respectively, we will require some additional notation as follows.
Let $\Gamma^\vee$ be the Dynkin diagram dual to $\Gamma$, and let $\alpha^\vee$ be the simple root in $\Gamma^\vee$ dual to the noncompact simple root $\alpha$ in the simple system $\Pi$ of $\Phi$.  Let $\Phi^\vee$ be the root system of $\Gamma^\vee$ with simple roots $\Pi^\vee$. Let $ \varIota^\vee=\Pi^\vee \setminus \{\alpha^\vee \}$. Define the \emph{simply laced cover} of $\Gamma^\vee$ to be the Dynkin diagram $\widehat{\Gamma}$ of type $\mathsf{A}_{2n-1}$ or $\mathsf{D}_{n+1}$, respectively.  Let $\widehat{\Phi}$ be the root system of $\widehat{\Gamma}$ with simple roots $\widehat{\Pi}$. Let $\widehat {\varIota}=\widehat{\Pi} \setminus \{\widehat{\alpha} \}$, where $\widehat{\alpha}=\alpha_1$ (resp.\ $\alpha_{n+1}$). Thus we have associated to $\Gamma=(\mathsf{B}_n,\mathsf{B}_{n-1})$ or $(\mathsf{C}_n, \mathsf{A}_{n-1})$, respectively, the simply laced Hermitian symmetric pair  $\widehat{\Gamma} = (\mathsf{A}_{2n-1}, \mathsf{A}_{n-1} \times \mathsf{A}_{n-1})$ or $(\mathsf{D}_{n+1}, \mathsf{A}_n)$, respectively.

\begin{thm}[{\cite[Thm.~7.2]{BH:09}}]\label{BC-Kostan}
Let $\mathcal{O}^{\mathfrak{q}}_{-2\rho}$ be the regular block for a Hermitian symmetric pair. Then there is a Dynkin diagram ${\mathcal{D}}$ with one crossed node $\alpha'$, such that there is a bijection 
$$
\{\,\textup{Kostant modules in }\mathcal{O}^{\mathfrak{q}}_{-2\rho}\,\} \longleftrightarrow 
\left\{
\begin{array}{c}
\textup{connected subdiagrams}\\
\textup{of } \mathcal{D} \textup{ containing } \alpha'\\  
\end{array}
\right\},
$$
where we declare the empty subdiagram to belong to the right-hand set.
If $\Gamma$ is simply laced, then $\mathcal{D}=\Gamma$ and $\alpha'=\alpha$; otherwise, $\mathcal{D}=\widehat{\Gamma}$ and $\alpha'=\widehat{\alpha}$.
\end{thm}

From the proof given in \cite[Thm.~7.2]{BH:09}, we have that $\Phi(\mathfrak{p}^+)$ and $\widehat{\Phi}(\widehat{\mathfrak{p}}^+)$ are isomorphic as posets. Since the elements in ${}^\mathfrak{k} W$ correspond to generalized Young diagrams from Proposition \ref{P:product}, which also correspond to lower order ideals of $\Phi(\mathfrak{p}^+)$, we find that  ${}^\mathfrak{k} W$ and ${}^{\widehat{\mathfrak{k}}} \widehat{W}$ are isomorphic as posets. Thus we have the following corollary.

\begin{cor}\label{kos-bij}
    With notation as in Theorem~\ref{BC-Kostan}, we have a bijection
    $$
\{\,\textup{Kostant modules in }\mathcal{O}^{\mathfrak{q}}_{-2\rho}\,\} \longleftrightarrow 
\left\{
\textup{Kostant modules in }\mathcal{O}^{\widehat{\mathfrak{q}}}_{-2\widehat{\rho}}
\right\}.
$$
\end{cor}

\subsection{Schubert varieties}	
	
Let $G$ be a complex simple algebraic group, $Q \supset B$ a parabolic subgroup containing the standard Borel, and $T \subset B$ the maximal torus.

\begin{dfn}
The \textit{generalized Schubert cell} associated to $w \in W^{\rm I}$ is
\[
C_{Q}(w) := B w Q \subset G/Q.
\]
Its Zariski closure is the \textit{Schubert variety} $X_Q(w) := \overline{C_Q(w)}$. When $Q=B$, we simply denote $X_B(w)$ by $X(w)$.
\end{dfn}

% Key properties:
% \begin{itemize}
%     \item $\dim_{\mathbb{C}} X(w) = \ell(w)$ (complex dimension equals length)
%     \item Cell decomposition: $X(w) = \bigcup_{v \leq w} C(v)$
%     \item Homology: $H_{2i}(X(w), \mathbb{Z}) \cong \mathbb{Z}^{\#\{v \leq w : \ell(v)=i\}}$
%     \item Poincaré polynomial: $p_w(t) = \sum_{v \leq w} t^{\ell(v)}$
% \end{itemize}

% \begin{figure}[h]
% \centering
% \begin{tikzpicture}[scale=0.6]
% \draw[thick] (0,0) grid (4,1);
% \draw[pattern=north west lines] (0,0) rectangle (1,1);
% \node at (0.5,0.5) {$X(w)$};
% \draw[thick] (6,0) grid (9,2);
% \draw[pattern=north west lines] (6,0) rectangle (7,1) (6,1) rectangle (7,2) (7,1) rectangle (8,2);
% \node at (7.5,1.5) {Smooth};
% \node[below] at (2,0) {$(A_4, A_2 \times A_1)$};
% \node[below] at (7.5,0) {Rational smoothness};
% \end{tikzpicture}
% \caption{Smooth Schubert varieties for Hermitian symmetric pair $(A_4, A_2 \times A_1)$. Rectangular Young diagrams correspond to smooth varieties.}
% \end{figure}

% When $G/P$ is a \textit{Hermitian symmetric space} (i.e., $\mathfrak{p}$ is of Hermitian type):
% \begin{itemize}
%     \item $G/P \cong K_{\mathbb{R}}/M_{\mathbb{R}}$ as compact Hermitian symmetric spaces
%     \item Bruhat and weak Bruhat orders coincide on $W^1$ (Corollary 3.12)
%     \item Schubert varieties correspond to generalized Young diagrams (Section 3.6)
% \end{itemize}

From Kazhdan--Lusztig \cite{KL79}, a Schubert variety is called \emph{rationally smooth}
whenever its certain cohomology groups are trivial. From Billey--Lakshmibai \cite[\S 4.2]{BL00}, a point $x\in X(w)$ is said to be a smooth point of $X(w)$ if the local ring $\mathcal{O}_{X(w),x}$ is a regular local ring;
moreover, $X(w)$ is called \emph{smooth} if any point of $X(w)$  is a smooth point.

\begin{prop}[{\cite[Thm.~E]{ca94}}]\label{ra-smooth}
The Schubert variety $X(w)$ is rationally smooth if and only if  its Poincar\'{e} polynomial $p_w(t) = \sum_{v \leq w} t^{\ell(v)}$ is palindromic.
\end{prop}

\begin{prop}[{\cite[Cor.~4]{Ca-11}}]\label{inverse}
    A Schubert variety $X(w)$ in $G/B$ is smooth if and only if the corresponding inverse Schubert variety $X(w^{-1})$ is also smooth.
\end{prop}

% \begin{proposition}[4.3, Carrell-Peterson]
% For $G/P$ Hermitian symmetric, $X(w)$ is rationally smooth iff its Poincaré polynomial $p_w(t)$ is \textit{palindromic} (satisfies Poincaré duality).
% \end{proposition}

% \begin{theorem}[4.9]
% Rational smoothness is equivalent to:
% \[
% P_{x,w}^1(q) = 1 \quad \forall x \leq w
% \]
% where $P_{\bullet,\bullet}^1$ are parabolic Kazhdan-Lusztig polynomials.
% \end{theorem}

% Combinatorial characterization for Hermitian symmetric spaces:
% \begin{itemize}
%     \item Type A: $X(w)$ smooth iff Young diagram is a rectangle
%     \item Types C/D: Smooth iff shifted Young diagram is single row or staircase
%     \item simply laced types: Rationally smooth $\iff$ smooth
% \end{itemize}
In simply laced types,  a Schubert variety in $G/Q$ is rationally smooth if and only if it is smooth~\cite[\S 4.3]{EHP}.
For any Hermitian symmetric space $G/Q$, the smooth Schubert varieties are also easy to describe.
Let $Q'\subset G$ be another parabolic subgroup containing $B$ and with  Levi decomposition
$Q'=LN$.
Then $L\cap B$ is a Borel subalgebra and $L\cap Q$ is a parabolic subgroup of the reductive group $L$ .  
Furthermore, $B=(L\cap B)N$. Now let $L$ act on $G/Q$ by left multiplication. Since the stabilizer
of $eQ$ is the parabolic subgroup $L\cap Q$, 
the orbit map $\varphi:L\rightarrow G/Q$, given by $x\mapsto x(eQ)=xQ$, induces a closed embedding
of the generalized flag variety $L/(L\cap Q)$ in $G/Q$.

\begin{lemma}[{\cite[Lemma 4.4]{EHP}}]
Under the closed embedding $L/(L\cap Q)\hookrightarrow G/Q$, the image of any $(L\cap B)$-orbit
in $L/(L\cap Q)$ is a $B$-orbit in $G/Q$. In particular, the image of $L/(L\cap Q)$ is a smooth Schubert variety
in $G/Q$. Furthermore, this Schubert variety is the  $Q'$-orbit of the coset $eQ$.
\end{lemma}

Suppose $Q$ is any (standard) maximal parabolic subgroup of $G$, not necessarily of Hermitian type.
For  $\mathrm{I} \subset \Pi$, let $Q_{\mathrm{I}}=L_{\,\mathrm{I}}N_{\,\mathrm{I}}$
be the standard parabolic subgroup corresponding to $\mathrm{I}$.
We say that  $\mathrm{I}$ is \emph{connected} if the corresponding subgraph of the Dynkin diagram of $G$ is connected.
Then we can define an injective map from the set  of connected subsets of $\Pi$ containing $\alpha$ to 
the set of smooth Schubert varieties in $G/Q$, via $\mathrm{I}\mapsto L_\mathrm{\,I}/(L_\mathrm{\,I}\cap Q)\hookrightarrow G/Q$.
%\footnote{Without the assumptions on $\mathrm{J}$, the map would still be defined, but it would not be injective
%(cf. Boe-Hunziker\cite{BoeHunziker:09}).} 

\begin{prop}[{\cite[Prop.  2.11]{Hong:07}}]\label{P:smooth}
Suppose that $G/Q$ is a Hermitian symmetric space. Let $\Gamma$ be the Dynkin diagram of $G$.
Then we have a bijection
$$
\left\{
\begin{array}{c}\textup{connected subdiagrams of $\ \Gamma$}\\
\textup{containing the noncompact simple root }
\end{array}
\right\}
\rightarrow \{\textup{smooth Schubert varieties in $G/Q$}\},
$$
given as above by 
$\mathrm{I}\mapsto L_\mathrm{\,I}/(L_\mathrm{\,I}\cap Q)\hookrightarrow G/Q$.
\end{prop}

% \begin{thm}[4.5, Hong]
% For $G/P$ Hermitian symmetric, smooth Schubert varieties are classified by connected subdiagrams of the Dynkin diagram containing the distinguished node:
% \[
% \left\{\!\!\begin{aligned} 
% \text{Connected subdiagrams}\\ 
% \text{containing marked node}
% \end{aligned}\right\} \xrightarrow{\sim} \left\{\!\!\begin{aligned} 
% \text{Smooth}\\ 
% \text{Schubert varieties}
% \end{aligned}\right\}
% \]
% via $J \mapsto L_J/(L_J \cap P) \hookrightarrow G/P$.
% \end{thm}

% \begin{example}[4.6]
% For $(C_3, A_2)$, 8 Schubert varieties exist: 6 rationally smooth, 4 smooth.
% \end{example}

\begin{thm}[{\cite[Thm.~5.15]{EHP}}] \label{T:BGG}
Let $L_w$ be a simple module in ${\mathcal O}^\fq_{-2\rho}$. For $w\in {}^\fk W$, the following are equivalent:
\begin{enumerate}
\item[{\rm (i)}]  $L_w$ is a Kostant module.
\item[{\rm (ii)}] The Schubert variety $X_Q(w^{-1})$ in $G/Q$ is rationally smooth.
\end{enumerate}
\end{thm}

\begin{rem}\label{rat-not-sm}
When $\mathfrak{g}$ is not simply laced, from Corollary \ref{kos-bij}  we know that the Kostant modules in $\mathcal{O}^{\mathfrak{q}}_{-2\rho}$  and $\mathcal{O}^{\tilde{\mathfrak{q}}}_{-2\rho}$ are in one-to-one correspondence. The  Kostant modules in $\mathcal{O}^{{\mathfrak{q}}}_{-2\rho}$ corresponding to connected subdiagrams of $\Gamma$ will correspond to smooth Schubert varieties by Proposition \ref{P:smooth}. Recall that the  Kostant modules in $\mathcal{O}^{\widehat{\mathfrak{q}}}_{-2\rho}$ correspond to connected subdiagrams of $\widehat{\Gamma}$ by Theorem \ref{sub-Kostant}. Thus the Kostant modules in $\mathcal{O}^{{\mathfrak{q}}}_{-2\rho}$ corresponding to  rationally smooth but not smooth Schubert varieties will correspond to subdiagrams of $\widehat{\Gamma}$ that do not correspond to connected subdiagrams of $\Gamma$.
\end{rem}
 
\begin{prop}[{\cite[Cor.~6.8]{RZ23}}]\label{bij-equi}
    Let $G/Q$ be a Hermitian symmetric space.
  Suppose $v\in W^{\fk}$. The Schubert variety $X_{Q}(v) \subset G/Q$ is smooth if and only if the Schubert variety $X({vw_{c}}) \subset G/B$ is smooth, where $w_{c}$ is the longest element in $W(\mathfrak{k})$. 
\end{prop}
  % \begin{proof}
  %     From \cite[Chap. 8 ]{BL00} or \cite{De85}, we know that the Schubert variety $X(w)$ is smooth if and only if the Schubert variety $X(y)$ is smooth for all $y\leq w$ in the Bruhat order. Let $w^c$ be the longest element of ${}^{\fk}W$, then we have $w_0=w_cw^c$ and $\Phi_{w^c}=\Phi(\fp^+)$. For any $x\in {}^{\fk}W$, we have $\Phi_x\subset \Phi_{w^c}$. Thus by Lemma \ref{L:ideal}, we have $x\leq w^c$, which implies that $w_cx\leq w_cw^c=w_0$. Therefore $X(w_cx)$ is smooth since $X(w_0)$ is smooth.????????
  % \end{proof}  
% From Proposition \ref{inverse}, for $v\in W^{\fk}$, we know that the Schubert variety $X({vw_{c}}) \subset G/B$ is smooth if and only if the Schubert variety $X({w_{c}v^{-1}}) \subset G/B$ is smooth since $w_{c}^{-1}=w_{c}$, equivalently the Schubert variety $X_{Q}(v) \subset G/Q$ is smooth by Proposition \ref{bij-equi}. 
% On the other direction, for $v\in {}^{\fk}W$, the Schubert variety $X({w_{c}v}) \subset G/B$ is smooth if and only if the Schubert variety $X({v^{-1}w_{c}}) \subset G/B$ is smooth by Proposition \ref{inverse}, equivalently the Schubert variety $X_{Q}(v) \subset G/Q$ is smooth by Theorem \ref{bij-equi}.

From Proposition \ref{inverse}, for $v\in {}^{\fk}W$, we know that the Schubert variety $X({w_{c}v}) \subset G/B$ is smooth if and only if the Schubert variety $X({v^{-1}w_{c}}) \subset G/B$ is smooth since $w_{c}^{-1}=w_{c}$, equivalently the Schubert variety $X_{Q}(v^{-1}) \subset G/Q$ is smooth by Proposition \ref{bij-equi}. 
%\textcolor{red}{I'll clarify this phrasing, after I make sure that I understand the statement.}

\begin{cor}\label{s-bij}
    Let $G/Q$ be a Hermitian symmetric space.
  Suppose $v\in {}^{\fk}W$. The Schubert variety $X_{Q}(v^{-1}) \subset G/Q$ is smooth if and only if the Schubert variety $X({w_{c}v}) \subset G/B$ is smooth, where $w_{c}$ is the longest element in $W(\mathfrak{k})$.
\end{cor}
% On the other direction, for $v\in W^{\rm J}$, the Schubert variety $X({vw_{\rm J}}) \subset G/B$ is smooth if and only if the Schubert variety $X({w_{\rm J}v^{-1}}) \subset G/B$ is smooth by Proposition \ref{inverse}, equivalently the Schubert variety $X_{P_{\rm J}}(v) \subset G/P_{\rm J}$ is smooth by Theorem \ref{bij-equi}.

\subsection{Unitary highest weight modules}\label{QR}
Let $L(\lambda)$ be a highest weight Harish-Chandra module. Let \(\beta\) be the unique maximal root in \(\Phi^{+}\). Choose \(\zeta \in \mathfrak{h}^{*}\) such that \(\zeta\) is orthogonal to \(\Phi(\mathfrak{k})\) and satisfies \( (\zeta, \beta^{\vee})= 1\).
Let \(\rho\) denote half the sum of  positive roots in \(\Phi^{+}\). We may write
\(
\lambda = \lambda_0 + z\zeta,
\)
where \(\lambda_0 \in \mathfrak{h}^{*}\) satisfies \( (\lambda_0 + \rho, \beta^{\vee}) = 0\), and \(z = ( \lambda + \rho, \beta^{\vee})  \in \mathbb{R}\).
Now we suppose that $\lambda=\lambda_0+z\zeta  \in \mathfrak{h}^{*}$ is $\Phi^+(\mathfrak{k})$-dominant integral. Define
$
Z(\lambda_0):=\{z\in \mathbb{R}\mid L(\lambda_0+z\zeta)\ \mbox{is unitarizable}\}.
$
By  Enright--Howe--Wallach~\cite[Thm.~2.4]{EHW}, the set $Z(\lambda_0)$ is given by the diagram shown in Figure~\ref{zk}.
\begin{figure}[H]
\centering
$
\begin{pspicture}(-1,-.5)(12,.2)
\cnode*(10,0){.07}{z1}
\cnode*(9,0){.07}{z2}
\cnode*(8,0){.07}{z3}
\cnode*(7,0){.07}{z4}
\cnode*(6,0){.07}{z5}
\cnode*(5,0){.07}{z6}
\psline[linewidth=2.0pt](-1,0)(5,0)
\psline[linewidth=0.5pt,arrowsize=4pt]{->}(-1,0)(12,0)
\psline[linewidth=0.5pt](0,-.1)(0,0.1)
\psline[linewidth=0.5pt](7,0)(7,-0.6)
\psline[linewidth=0.5pt](8,0)(8,-0.6)
\psline[linewidth=0.5pt]{->}(6.7,-.4)(7,-0.4)
\psline[linewidth=0.5pt]{->}(8.3,-.4)(8,-0.4)
\uput[d](10,-.05){$b(\lambda_0)$}
\uput[d](7.5,-.1){$c$}
\uput[d](5,-.05){$a(\lambda_0)$}
\uput[d](0,-.05){$0$}
\uput[r](12,0){$z$}
%\uput[d](6,-1){{\sc Figure 1}: The set $Z(\tau)$}
\end{pspicture}
$
\caption{The set $Z(\lambda_0)$}\label{zk}
\end{figure}
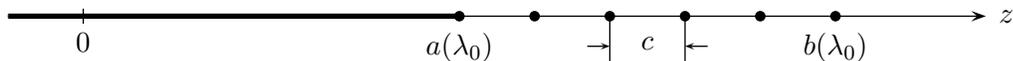

\noindent Here, $a(\lambda_0)$ is the so-called first unitary reduction point and $b(\lambda_0)$ is the last unitary reduction point. Both of these reduction points depend on certain root systems associated with $\lambda_0$; see \cite{EHW} or \cite{BH}. The set  $Z(\lambda_0)$ includes the ray ending at $a(\lambda_0)$, as well as certain reduction points between $a(\lambda_0)$ and $b(\lambda_0)$ that are spaced at an interval of length $c$, whose value (from~\cite{EHW}) can be found in  Table~\ref{constants-k}.
\begin{center}
		\begin{table}[ht]
        
			{\begingroup

            \renewcommand{\arraystretch}{1.5}
\setlength\tabcolsep{15pt}
            \begin{tabular}
    {c|c|c|c}
					%\begin{tabular}{|l|l|l|l|}
					\hline
					$\mathfrak{g}_{\R}$ &   $r$ & $c$ & $h^\vee$ \\  
					\hline  
					$\mathfrak{su}(p,n-p)$ & $\min\{p, n-p\}$ &  $1$ & $n$  \\  \hline
					$\mathfrak{sp}(n,\mathbb{R})$  & $n$ &   $\frac{1}{2}$ & $n+1$   \\  \hline 
					$\mathfrak{so}^{*}(2n)$  & $[\frac{n}{2}]$  & $2$ & $2n-2$ \\ \hline
					$\mathfrak{so}(2,2n-1)$  & $2$  &  $n-\frac{3}{2}$ & $2n-1$ \\ \hline 
					$\mathfrak{so}(2,2n-2)$  & $2$ &  $n-2$& $2n-2$ \\ \hline
					$\mathfrak{e}_{6(-14)}$  & $2$ &  $3$& $12$ \\ \hline
					$\mathfrak{e}_{7(-25)}$  & $3$ &  $4$& $18$ \\ %\hline
					\hline
				\end{tabular}
                \endgroup
			}\caption{Some constants of  Lie groups of Hermitian type}\label{constants-k}
		\end{table}
	\end{center}

% When $\lambda_0=-(\rho,\beta^\vee)\zeta$, we will have $a(\lambda_0)=z_{r-1}=(\rho,\beta^\vee)-(r-1)c$ and $b(\lambda_0)=z_0=(\rho,\beta^\vee)$. We note that it is quite surprising that for a    general $\lambda_0$, we have $$Z(\lambda_0)\subseteq Z(-(\rho,\beta^\vee)\zeta)=(-\infty,z_{r-1})\cup \{z_{r-1},\ldots,z_0\}
% $$
% from \cite[Lemma 6.1]{BH}.

Now we recall two root systems
 $\rQ(\lambda_{0})$ and $ \rR(\lambda_{0})$ given in \cite{EHW}, which are constructed as follows.  Let $\Phi_{c}(\lambda_{0})=\{\alpha \in \Phi(\mathfrak{k})| (\lambda_{0}, \alpha)=0\}$. Consider the root subsystem $\Psi_1$ of $\Phi$, which is generated by $\pm \beta$ and $\Phi_{c}(\lambda_{0})$. Let  $\rQ(\lambda_{0})$ be the simple component of  $\Psi_1$ which contains $-\beta.$ If $\Phi$ has two root lengths  and there exist short  roots $\alpha'\in \Phi(\mathfrak{k})$ that are not orthogonal to $\rQ(\lambda_{0})$ and satisfy $(\lambda_{0},\alpha'^{\vee})=1$, then let $\Psi_2$ be the root system generated by $\pm\beta$, $\Phi_{c}(\lambda_{0})$, and all such $\alpha$.
Let $\rR(\lambda_{0})$ be the simple component of $\Psi_2$ which contains $-\beta$. If $\Phi$ has only one root length or no such $\alpha$ exists, then we let $\rR(\lambda_{0})=\rQ(\lambda_{0})$. Note that $\Phi_{c}(\lambda_{0})=\Phi_{c}(\lambda)$ since $(\lambda_0,\alpha)=(\lambda,\alpha)$. Thus we have $\rQ(\lambda_{0})=\rQ(\lambda)$	 and $\rR(\lambda_{0})=\rR(\lambda)$.	 
For a root system $\Phi$, we write $h_{\Phi}^{\vee}=(\rho,\beta^{\vee})+1$
 to denote the dual Coxeter number of $\Phi$; see Table~\ref{constants-k}.
 
\begin{thm}[{\cite[Thm. ~3.2]{BH2025}}]\label{last-unitary}
The last unitary reduction point $b(\lambda_0)$ is given by $$b(\lambda_0)=h_{\rQ(\lambda_0)}^\vee - 1 + \frac{r_{\rR(\lambda_0)} - r_{\rQ(\lambda_0)}}{2},$$
where $r_{\rQ(\lambda_0)}$ and $r_{\rR(\lambda_0)}$ denote the split rank of $\rQ(\lambda_0)$ and $\rR(\lambda_0)$ respectively.
    
\end{thm}

From Enright--Howe--Wallach \cite[\S 5]{EHW} or Enright--Joseph \cite[\S 1.4]{EJ}, the split rank $r$ of  the root system $\Phi$ of $\mathfrak{g}_{\mathbb{R}}$ is equal to the number of strongly orthogonal positive noncompact roots in $\Phi(\mathfrak{p}^{+})$; see Table~\ref{constants-k}.

\begin{thm}[{\cite[Thm~2.2]{Enright}}]\label{kostant-module}
 If $L(\lambda)$ is a unitary highest weight module, then it is a Kostant module. 
%Moreover, its $\mathfrak{u}$-cohomology is given by:
% \[
% H^i(\mathfrak{u}, L(\lambda)) \cong \bigoplus_{\substack{y \in {}^E(W_{[\lambda]}) \\ \ell_{[\lambda]}(y) = i}} F_{\bar{y} \cdot \lambda}
% \]
% where ${}^E(W_{[\lambda]})$ is the Weyl group of the \textit{reduced Hermitian pair} associated to $\lambda$.
\end{thm}
\subsection{Associated varieties}

In \cite{Vo91}, Vogan introduced the notion of an associated variety for a Harish-Chandra module $M$ of a real reductive Lie group $G_{\mathbb{R}}$.
Let $M$ be a finitely generated $U(\mathfrak{g})$-module. Fix a finite dimensional generating space $M_0$ of $M$. Let $U_{n}(\mathfrak{g})$ be the standard filtration of $U(\mathfrak{g})$. Set $M_n=U_n(\mathfrak{g})\cdot M_0$ and
	\(
	\text{gr} (M)=\bigoplus_{n=0}^{\infty} \text{gr}_n M,
	\)
	where $\text{gr}_n M=M_n/{M_{n-1}}$. Thus $\text{gr}(M)$ is a graded module of $\text{gr}(U(\mathfrak{g}))\simeq S(\mathfrak{g})$.

The  \textit{associated variety} of $M$ is defined by
		\begin{equation*}
		\AV(M):=\{X\in \mathfrak{g}^* \mid f(X)=0 \text{ for all~} f\in \operatorname{Ann}_{S(\mathfrak{g})}(\operatorname{gr} M)\}.
		\end{equation*}	
Note that this  definition is independent of the choice of $M_0$ (see~\cite{NOT}, for example). 
	If $M_0$ is $\mathfrak{a}$-invariant for a subalgebra $\mathfrak{a}\subset\mathfrak{g}$, then 
	\begin{equation}\label{embed}
	\AV(M)\subset (\mathfrak{g}/\mathfrak{a})^*.
	\end{equation}
In general, the associated
variety $\AV(M)$ is a finite union of the closure of nilpotent $K$-orbits. 
From \cite{NOT}, we know that the closures of the $K$-orbits in $\mathfrak{p}^+$ form a linear chain of  varieties
	\begin{equation*}\label{chain}
	\{0\}={\overline{\mathcal{O}_0}}\subset \overline{\mathcal{O}_1}\subset \cdots \subset\overline{\mathcal{O}_{r-1}}\subset \overline{\mathcal{O}_{r}}=\mathfrak{p}^+,
	\end{equation*}
	where $r$ is the split rank of $\mathfrak{g}_{\mathbb{R}}$ or equivalently $\mathbb{R}$-rank of  $G_{\mathbb{R}}$, i.e., the dimension of  a Cartan subgroup of the  group $G_{\mathbb{R}}$, which is also equal to the rank of the symmetric space $G_{\mathbb{R}}/K_{\mathbb{R}}$.

	When $M=L(\lambda)$ is a highest weight Harish-Chandra module, we can choose $L_0$ to be the finite dimensional $U(\mathfrak{k})$-module generated by $\mathbb{C}_{\lambda}$. Then $L_0$ is $\mathfrak{k}\oplus\mathfrak{p}^+$-invariant. In view of \eqref{embed},
	\[
	\AV(L(\lambda))\subset(\mathfrak{g}/(\mathfrak{k}\oplus\mathfrak{p}^+))^*\simeq(\mathfrak{p}^-)^*\simeq \mathfrak{p}^+,
	\]
	where the last isomorphism is induced from the Killing form. As shown in \cite{Vo91}, the associated variety $\AV(L(\lambda))$ is also $K$-invariant. In fact, Yamashita \cite{Yam-01} proved that $ \AV(L(\lambda)) $ must be one of the orbit closures $ \overline{\mathcal{O}_{k}}$.
	
	\begin{prop}[{\cite[\S 3.2]{Yam-01}}]\label{Ok}
		Let $L(\lambda)$ be a highest weight Harish-Chandra module. Then 
		\begin{equation*}
		\AV(L(\lambda))=\overline{\mathcal{O}_{k(\lambda)}}
		\end{equation*}
		for some $0\leq k(\lambda)\leq r$.
	\end{prop}

	Define \begin{equation}\label{z-k}
	   z_{k}:=(\rho,\beta^{\vee})-kc, 
	\end{equation}
 for $0\leq k\leq r$, where $r$ is the $\mathbb{R}$-rank of $G_{\mathbb{R}}$ and   $c$ is the constant associated with the Hermitian type Lie group $G_{\mathbb{R}}$ in Table~\ref{constants-k}.   %(hereafter referred to as EHW\cite{EHW}):
In \cite{BH}, there is  a uniform expression for the GK dimensions and associated varieties of unitary highest weight Harish-Chandra modules.
	
	\begin{prop}[{\cite{BH}}]\label{C: dimYk}
		Suppose $L(\lambda)$ is a unitary highest weight Harish-Chandra module with highest weight $\lambda$. Writing $z:=z(\lambda)=(\lambda+\rho,\beta^{\vee})$, we have
		\begin{align*}\gk L(\lambda)=
		\begin{cases}
		rz_{r-1} & \textup{if  $z<z_{r-1}=(\rho,\beta^{\vee})-(r-1)c$},\\
		\ell z_{\ell -1} & \textup{if  $z=z_{\ell}=(\rho,\beta^{\vee})-\ell c,$ for $1\leq \ell \leq r-1$},\\
		0 & \textup{if $z=z_{0}=(\rho,\beta^{\vee})$}.
		\end{cases}
		\end{align*}
		
		Define $k=k(\lambda):=-\frac{(\lambda,  \beta^{\vee} )}{c}$.
        Then
		\begin{enumerate}
			\item  If $k>r-1$, we have $\gk L(\lambda)=rz_{r-1}=\frac{1}{2}\dim(G/K).$
			\item If $0\leq k\leq r-1$, then $k$ is a nonnegative integer and $$\gk L(\lambda)=k((\rho, \beta^{\vee})-(k-1)c)=kz_{k-1}=\dim \overline{\mathcal{O}_{k(\lambda)}}.$$
		\end{enumerate}
		The associated variety of $L(\lambda)$ is $\overline{\mathcal{O}_{k(\lambda)}}$.
	\end{prop}

\section{Uniform proof of the bijection for simply laced types}\label{uniform-proof}
In this section, we will give a uniform proof for the unitarity of Kostant modules corresponding to  connected Dynkin subdiagrams that
contain the noncompact simple root. This will lead to the bijection in Theorem \ref{thm:bij} for all simply laced cases.

Let \( \varIota := \Pi \setminus \{ \alpha \} \), where \( \alpha \) is the noncompact  simple root. Let \( \Gamma \) be the corresponding Dynkin diagram of the root system $\Phi$ with the noncompact root $\alpha$.
Suppose $\Phi' \subseteq \Phi$ is a root subsystem corresponding to a connected Dynkin subdiagram $\Gamma'$ that contains the noncompact simple root. Then $\Phi' \cap \Phi(\mathfrak{p}^+)$ is a lower order ideal in $\Phi(\mathfrak{p}^+)$, and hence
\begin{equation*}
\Phi' \cap \Phi(\mathfrak{p}^+) = \Phi_x
\end{equation*}
for some $x \in {}^\mathfrak{k}W$, by Lemma~\ref{L:ideal}. Define
\begin{equation*}
\lambda_{\Phi'} := -w_c x \rho - \rho,
\end{equation*}
where $w_c$ denotes the longest element of the compact Weyl group $W(\mathfrak{k})$.
Recall that the dual Coxeter number of a root system $\Phi$ is $h^{\vee}:=(\rho,\beta^{\vee})+1$, which we display in Table~\ref{constants-k}.
\begin{lemma}\label{coxe}
For every $\Phi' \subseteq \Phi$ as above, we have
\begin{equation*}
(\lambda_{\Phi'} + \rho, \beta^\vee) = h'^\vee - 1,
\end{equation*}
where $h'^\vee$ is the dual Coxeter number of the root system $\Phi'$.
\end{lemma}

\begin{proof}
Fix $\Phi' \subset \Phi$ as above, and let $x \in {}^\mathfrak{k}W$ such that $\Phi' \cap \Phi(\mathfrak{p}^+) = \Phi_x$.
Let
\begin{equation*}
\tilde{x} := w_c x w_0,
\end{equation*}
where $w_0$ is the longest element in the full Weyl group $W$. Since $w_0 \rho = -\rho$, it follows that $\lambda_{\Phi'} = \tilde{x} \rho - \rho$, and hence by Lemma \ref{L:rho}, we have
\begin{equation}\label{rootsum}
\lambda_{\Phi'} = -\langle \Phi_{\tilde{x}} \rangle.
\end{equation}
% Here and throughout, for any subset $\Psi \subseteq \Phi$, we write $\langle \Psi \rangle$ to denote the sum $\sum_{\psi \in \Psi} \psi$.
By Enright--Hunziker--Pruett \cite[Cor.~3.15]{EHP} or Lemma \ref{L:involution}, we have
\begin{equation*}
\Phi_{\tilde{x}} = \Phi(\mathfrak{p}^+) \setminus w_c \Phi_x,
\end{equation*}
and hence
\begin{equation}\label{eq: lambda decomp}
\lambda_{\Phi'} = -\langle \Phi(\mathfrak{p}^+) \rangle + w_c \langle \Phi_x \rangle.
\end{equation}

By Panyushev \cite[Thm.~2.2]{Pan20}, 
\begin{equation*}
\langle \Phi(\mathfrak{p}^+) \rangle = h^\vee \zeta,
\end{equation*}
where $h^\vee$ is the dual Coxeter number of the root system $\Phi$. Similarly, since $\Phi' \cap \Phi(\mathfrak{p}^+) = \Phi_x$,  
\begin{equation*}
\langle \Phi_x \rangle = h'^\vee \zeta',
\end{equation*}
where $h'^\vee$ is the dual Coxeter number of the root system $\Phi'$, and 
$\zeta'$ is the fundamental weight of $\Phi'$ corresponding to the noncompact simple root, viewed as an element of $\mathfrak{h}^*$.
(\textit{Warning: $\zeta'$ is, in general, \emph{not} a fundamental weight for the full root system $\Phi$.})
Thus, by \eqref{eq: lambda decomp}, we can write
\begin{equation}\label{eq-lambda}
\lambda_{\Phi'} = -h^\vee \zeta + h'^\vee w_c \zeta'.
\end{equation}
Now recall that $w_c \beta = \alpha$, since $\mathfrak{p}^+$ is a highest weight module of $\mathfrak{k}$ with highest weight $\beta$ and lowest weight $\alpha$. Therefore,
\begin{equation}\label{zeprime}
(w_c \zeta', \beta^\vee) = (\zeta', w_c \beta^\vee) = (\zeta', \alpha^\vee) = 1.
\end{equation}
Since $(\zeta, \beta^\vee) = 1$ and $(\rho, \beta^\vee) = h^\vee - 1$, it follows that
\begin{align*}
(\lambda_{\Phi'} + \rho, \beta^\vee) &= (-h^\vee \zeta + h'^\vee w_c \zeta' + \rho, \beta^\vee) \\
&= -h^\vee (\zeta, \beta^\vee) + h'^\vee (w_c \zeta', \beta^\vee) + (\rho, \beta^\vee) \\
&= -h^\vee + h'^\vee + (h^\vee - 1) \\
&= h'^\vee - 1,
\end{align*}
as claimed.
\end{proof}

We define $w(\Gamma'):=w(\Phi')=\tilde{x}=w_cxw_0$. Thus $w_cx=w(\Gamma')w_0$ and $\lambda_{\Phi'}=-w(\Gamma')w_0\rho-\rho=w(\Gamma')\rho-\rho$.

\begin{prop}\label{unitarity}
For every $\Phi' \subseteq \Phi$ as above, the simple module $L({\lambda_{\Phi'}})$ is unitarizable.
\end{prop}

\begin{proof} Fix $\Phi' \subseteq \Phi$ as above, and let $\rQ(\lambda_{\Phi'})$ and $\rR(\lambda_{\Phi'})$ be the two root subsystems of $\Phi$ associated with the line $\lambda_{\Phi'} + \mathbb{R} \zeta$, as in \S \ref{QR}. We claim that
\begin{equation} \label{eq:Q_equal}
\rQ(\lambda_{\Phi'}) = \rR(\lambda_{\Phi'}) = w_c \Phi'.
\end{equation}

To prove the claim, we must consider more closely the weight $\zeta' \in \mathfrak{h}^*$ introduced in the proof of the previous Lemma \ref{coxe}. Let $\{\alpha_1, \ldots, \alpha_n\}$ be the simple roots of $\Phi$, ordered according to Bourbaki \cite{Bour}, and set
\begin{equation*}
I = \{i : \alpha_i \in \Phi'\}.
\end{equation*}
By the definition of $\zeta'$, for all $i \in I$ we have
\begin{equation*}
(\zeta', \alpha_i^\vee) =
\begin{cases}
0 & \text{if $\alpha_i$ is compact,} \\
1 & \text{if $\alpha_i$ is noncompact.}
\end{cases}
\end{equation*}
Furthermore, by \cite[Thm.~2.2]{Pan20}, we have
\begin{equation*}
h' \zeta' = \left\langle \Phi' \cap \Phi(\mathfrak{p}^+) \right\rangle = \sum_{i \in I} n_i \alpha_i, \quad \text{with } n_i > 0 \text{ for all } i \in I.
\end{equation*}

Suppose $j \notin I$ is such that $\alpha_j$ is adjacent to the Dynkin diagram of $\Phi'$. (Then $\alpha_j$ is necessarily compact.) Since the Dynkin diagram of $\Phi'$ is connected (and the Dynkin diagram of $\Phi$ is a tree), there exists a unique $i_0 \in I$ such that
$(\alpha_{i_0}, \alpha_j^\vee) \ne 0$. In  fact, $(\alpha_{i_0}, \alpha_j^\vee) < 0$ and it  follows that
\begin{equation}\label{eq:negative}
(h' \zeta', \alpha_j^\vee) = (\alpha_{i_0}, \alpha_j^\vee) n_{i_0} < 0.
\end{equation}

Recall that $\lambda_{\Phi'} = -h^\vee \zeta + h'^\vee w_c \zeta'$ by (\ref{eq-lambda}).
Since $\zeta$ is orthogonal to all compact simple roots, it follows from this expression that every compact simple root $\alpha_i$ with  $i\in I$ lies in $w_c \Phi'$. (Indeed, each compact simple root is of the form $-w_c \alpha_i$ for some $i \in I$ with $\alpha_i$ compact.)
By the definition of the root subsystem $\rQ(\lambda_{\Phi'})$, we then have $w_c \Phi' \subseteq \rQ(\lambda_{\Phi'})$.
Note that the set $\{-w_c \alpha_i : i \in I\}$ forms a system of simple roots for $w_c \Phi'$, and its Dynkin diagram can be identified with a connected subdiagram of the extended Dynkin diagram of $\Phi$ containing the node corresponding to $-\beta= -w_c \alpha$, where $\alpha$ is the simple noncompact root of $\Phi$ (and also of $\Phi'$). A simple compact root not in $w_c \Phi'$ is adjacent to this Dynkin subdiagram if and only if it is of the form $-w_c \alpha_j$, where $\alpha_j$ is a simple compact root adjacent to the Dynkin diagram of $\Phi'$. In that case, by \eqref{eq:negative} we have
\begin{equation*}
(\lambda_{\Phi'}, -w_c \alpha_j) > 0.
\end{equation*}
Thus, by the definition of $\rQ(\lambda_{\Phi'})$, we conclude that
\begin{equation*}
w_c \Phi' = \rQ(\lambda_{\Phi'}).
\end{equation*}
If $\Phi$ is simply laced, then $\rQ(\lambda_{\Phi'}) = \rR(\lambda_{\Phi'})$, and the claim~\eqref{eq:Q_equal} follows. Also, if $\Phi' = \Phi$, then we  trivially have $\rQ(\lambda_{\Phi'}) = \rR(\lambda_{\Phi'})$.
So suppose $\Phi$ is not simply laced and $\Phi' \subsetneq \Phi$. We must consider two cases: $\Sp(2n, \mathbb{R})$ and $\SO(2,2n-1)$.

First we consider the case $\Sp(2n, \mathbb{R})$. In this case, for every $\Phi' \subsetneq \Phi$, we have $n_i > 1$ for all $i \in I$ such that $\alpha_i$ is compact. Thus, the argument above shows that
\begin{equation*}
(\lambda_{\Phi'}, -w_c \alpha_j) > 1.
\end{equation*}
Hence, by the definition of $\rR(\lambda_{\Phi'})$, it follows that $\rQ(\lambda_{\Phi'}) = \rR(\lambda_{\Phi'})$.

Next we consider the case $\SO(2,2n-1)$. Suppose first that $\Phi'$ is of type $\mathfrak{su}(1,n')$ with $n' \leq n-2$. Then the (unique) simple compact root adjacent to the Dynkin diagram of $\Phi'$ is long. By the definition of $\rR(\lambda_{\Phi'})$, it follows that $\rQ(\lambda_{\Phi'}) = \rR(\lambda_{\Phi'})$ in this case.

Now suppose $\Phi'$ is of type $\mathfrak{su}(1,n-1)$. Then the (unique) simple compact root adjacent to the Dynkin diagram of $\Phi'$ is the short simple root $\alpha_j = \alpha_n$, which is adjacent to the long simple root $\alpha_{i_0} = \alpha_{n-1}$. Since $(\alpha_{i_0}, \alpha_j^\vee) = -2$, we find that $(\lambda_{\Phi'}, -w_c \alpha_j) > 1$. By the definition of $\rR(\lambda_{\Phi'})$, it follows that $\rQ(\lambda_{\Phi'}) = \rR(\lambda_{\Phi'})$ in this case as well.

Thus, the claim \eqref{eq:Q_equal} holds in all cases. As in \S\ref{QR}, we can write $\lambda_{\Phi'}=\lambda_0+z\zeta$, with $\lambda_{0} \in \mathfrak{h}^{*} $ such that ($\lambda_{0}  + \rho, \beta^\vee  $)=0, and $z=(\lambda_{\Phi'}+\rho,\beta^\vee) \in \mathbb{R}$. By  Lemma \ref{coxe} and \eqref{eq:Q_equal}, we then have
\begin{equation*}
z=(\lambda_{\Phi'} + \rho, \beta^\vee) = h'^\vee - 1 = h_{\rQ(\lambda_{\Phi'})}^\vee - 1 + \frac{r_{\rR(\lambda_{\Phi'})} - r_{\rQ(\lambda_{\Phi'})}}{2}=b(\lambda_0).
\end{equation*}
By Theorem \ref{last-unitary}, it follows that $L(\lambda_{\Phi'})$ is unitarizable. (More precisely, $\lambda_{\Phi'}$ is the ``last place of unitarity'' on the line $\lambda_{\Phi'} + \mathbb{R} \zeta$.) This completes the proof of the proposition.
\end{proof}

% \begin{cor}
%     For any $w\in {}^\mathfrak{k}W $, $L({\tilde{w}}\rho-\rho)=L(w_cww_0\rho-\rho)=L(-w_cw)=L(-w_cw\rho-\rho)$ is a highest weight Harish-Chandra module.
% \end{cor}

\begin{cor}\label{bij-simply}
    If $G_{\mathbb{R}}$ is a Lie group of Hermitian type whose complexified Lie algebra $\mathfrak{g}$ is simply laced, then we have the bijection given in Theorem \ref{thm:bij}.
\end{cor}
\begin{proof}
  Suppose that $\Gamma'$ is a connected Dynkin subdiagram 
containing the noncompact simple root $\alpha$. Denote the corresponding root subsystem by $\Phi'$.  From Proposition \ref{unitarity}, we know that the highest weight module $L(\lambda_{\Phi'})=L(-w_cx\rho-\rho)=L_x$ is unitarizable.
By Theorem \ref{kostant-module}, we know that $L(\lambda_{\Phi'})=L_x$ is a Kostant module. By Theorem \ref{T:BGG}, the Schubert variety $X_Q(x^{-1})$ is smooth, since rationally smoothness and smoothness are equivalent for all simply laced types. By Proposition \ref{P:smooth}, there is a bijection between the smooth Schubert varieties and connected subdiagrams containing the noncompact simple root $\alpha$. Therefore we obtain the bijection in Theorem \ref{thm:bij}.
\end{proof}

Recall that  $L(-w\rho-\rho)$ is a highest weight Harish-Chandra module if and only if $w=w_cx$ for some $x\in {}^{\fk}W$.
Thus from  the proof of Corollary \ref{bij-simply}, we have the following result.
\begin{cor}\label{simply-schu}
Let $G_{\mathbb{R}}$ be a Lie group of Hermitian type whose complexified Lie algebra $\mathfrak{g}$ is simply laced. Suppose $L(-w\rho-\rho)$ is a highest weight Harish-Chandra module of $G_{\mathbb{R}}$. 
  Then  $L(-w\rho-\rho)$ is  unitarizable  if and only if the Schubert variety $X_Q(x^{-1})$ in $G/Q$ is smooth, where $x=w_cw$, and where $Q$ is the parabolic subgroup of $G$ such that ${\rm Lie}(Q)=\fq=\fk+\fp^+$.
\end{cor}

\section{Proof of the bijection for non-simply laced types}\label{case-proof}

In this section, we give the proof of the bijection in Theorem \ref{thm:bij} for non-simply laced cases, by using the results about Kostant modules from Boe--Hunziker \cite{BH:09}.

\begin{prop}\label{generalizeddiagram}
     If $G_{\mathbb{R}}$ is a Lie group of Hermitian type whose complexified Lie algebra $\mathfrak{g}$ is non-simply laced, then we have the bijection given in Theorem \ref{thm:bij}.
\end{prop}

\begin{proof}
Recall that if $\mathfrak{g}$ is non-simply laced, then the corresponding  Dynkin diagram for the Hermitian symmetric pair is $\Gamma=(\mathsf{B}_n, \mathsf{B}_{n-1})$ or $(\mathsf{C}_n, \mathsf{A}_{n-1})$, to which we associated the simply laced Hermitian symmetric pair $\widehat{\Gamma} = (\mathsf{A}_{2n-1}, \mathsf{A}_{n-1} \times \mathsf{A}_{n-1})$ or $(\mathsf{D}_{n+1}, \mathsf{A}_n)$, respectively.
The Kostant modules are characterized in Theorem \ref{BC-Kostan}. 
    
    First we consider the case $\SO(2,2n-1)$. In this case, $\Gamma=(\mathsf{B}_n, \mathsf{B}_{n-1})$. By Theorem \ref{BC-Kostan}, there is a bijection between the Kostant modules and the connected subdiagrams
of $\widehat{\Gamma}=(\mathsf{A}_{2n-1}, \mathsf{A}_{n-1} \times \mathsf{A}_{n-1})$ containing $\alpha'=\alpha_1=e_1-e_2$. The number of these subdiagrams is
$2n$.
By Proposition \ref{P:smooth},  there is a bijection between the smooth Schubert varieties in $G/P$ and the connected subdiagrams of $\Gamma$
containing the noncompact simple root $\alpha=e_1-e_2$. The number of these subdiagrams is $n+1$.
Thus there are $n-1$ many Kostant modules which correspond to rationally smooth (not smooth) Schubert varieties in $G/Q$. By Remark \ref{rat-not-sm}, each of these Kostant modules corresponds to a subdiagram of $\widehat{\Gamma}=(\mathsf{A}_{2n-1}, \mathsf{A}_{n-1} \times \mathsf{A}_{n-1})$ that does not correspond to any subdiagram of $\Gamma=(\mathsf{B}_n, \mathsf{B}_{n-1})$. By inspecting the corresponding generalized Young diagrams, we find that each of these subdiagrams is isomorphic to $\mathfrak{su}(1,p)$, for some $p$ such that $n \leq p \leq 2n-2$:
%     \begin{pspicture}(-.10,-.1)(2.6,.1)
% \cnode*(0,0){.07}{d1}
% \uput[d](0,0){\small $1$}
% \cnode*(0.5,0){.07}{d2}
% %\uput[d](0.5,0){\small $n-p+1$}
% \cnode*(1 ,0){.07}{d3}
% \uput[d](1.5,-.065){\small $\cdots$}
% \cnode*(1.5,0){.07}{d4}
% \cnode*(2,0){.07}{d5}
% %\cnode*(2.354,0.354){.07}{d6}
% %\uput[r](2.1,0.2){\small $n-2$}
% \cnode*(2.5,0){.07}{d7}
% \uput[r](2.3,-0.354){\small $p$}
% \ncline{d1}{d2}
% \ncline[linestyle=dotted,dotsep=1.3pt]{d2}{d3}
% \ncline{d3}{d4}
% \ncline{d4}{d5}
% %\ncline{d5}{d6}
% \ncline{d5}{d7}
% \pscircle[linewidth=.5pt,fillstyle=solid](0,0){.07}
% \pscircle[linewidth=.5pt,fillstyle=solid](0.5,0){.07}
% \pscircle[linewidth=.5pt,fillstyle=solid](1,0){.07}
% \pscircle[linewidth=.5pt,fillstyle=solid](1.5,0){.07}
% \pscircle[linewidth=.5pt,fillstyle=solid](2,0){.07}
% %\pscircle[linewidth=.5pt,fillstyle=solid](2.354,.354){.07}
% \end{pspicture}
\[
\begin{pspicture}(-.10,-.1)(2.6,.1)
\cnode*(0,0){.07}{d1}
\uput[d](0,0){\tiny $1$}
\cnode*(0.5,0){.07}{d3}
\uput[d](0.5,0){\tiny $2$}
\cnode*(1 ,0){.07}{d4}
\cnode*(1.5,0){.07}{d5}
\uput[d](1.5,-.065){\small $\cdots$}
\cnode*(2,0){.07}{d6}
\cnode*(2.5,0){.07}{d7}
\uput[d](2.5,-.065){\tiny $p$}
\ncline{d1}{d3}
\ncline[linestyle=dotted,dotsep=1.3pt]{d3}{d4}
\ncline{d4}{d5}
\ncline{d5}{d6}
\psline[linewidth=.5pt](2,0)(2.5,0)
%\psline[linewidth=.5pt](2,0)(2.5,0)
%\uput[r](1.925,0){$>$}
\pscircle[linewidth=.5pt,fillstyle=solid](0.5,0){.07}
\pscircle[linewidth=.5pt,fillstyle=solid](1,0){.07}
\pscircle[linewidth=.5pt,fillstyle=solid](1.5,0){.07}
\pscircle[linewidth=.5pt,fillstyle=solid](2,0){.07}
\pscircle[linewidth=.5pt,fillstyle=solid](2.5,0){.07}
\end{pspicture}.
\]
From \cite[\S 3.7]{EHP}, the  generalized Young diagram for the corresponding $x\in {}^\mathfrak{k}W$ is 
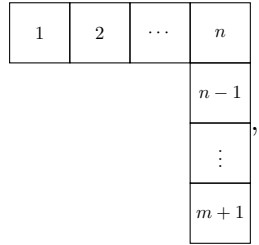
\begin{equation*}
    \small{ \begin{tikzpicture}[scale=0.8,baseline=-48pt]
		\hobox{0}{0}{1}
		\hobox{1}{0}{2}
            \hobox{2}{0}{\cdots}
		\hobox{3}{0}{n}
        \hobox{3}{1}{n-1}
        \hobox{3}{2}{\vdots}
        \hobox{3}{3}{m+1}
  		\end{tikzpicture}},
\end{equation*}
where $m=2n-p-1$.
Correspondingly we have 
$$\Phi_x=\tiny
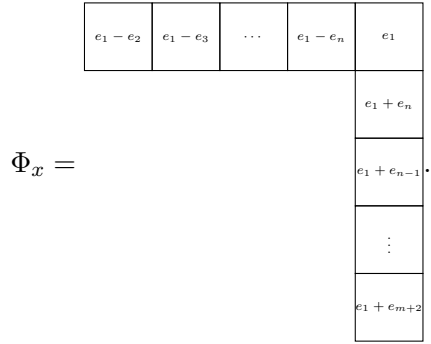
{ \begin{tikzpicture}[scale=0.9,baseline=-64pt]
		\hobox{0}{0}{e_1-e_2}
		\hobox{1}{0}{e_{1}-e_3}
            \hobox{2}{0}{\cdots}
             \hobox{3}{0}{e_1-e_n}
		\hobox{4}{0}{e_1}
         \hobox{4}{1}{e_1+e_n}
          \hobox{4}{2}{e_1+e_{n-1}}
           \hobox{4}{3}{\vdots}
            \hobox{4}{4}{e_1+e_{m+2}}
  		\end{tikzpicture}}.$$
The corresponding Kostant module is $L(\lambda)$ with highest weight $\lambda=\tilde{x}\rho-\rho$.
By Lemma \ref{L:rho}, we have $\lambda=-\langle \Phi_{\tilde{x}}\rangle$. By Lemma \ref{L:involution}, we have
% $ \Phi_{\tilde{x}}=w_c^{-1}(\Phi(\mathfrak{p}^{+})\setminus \Phi_x)$ {\color{red}as follows:}
$$\Phi_{\tilde{x}}=w_c^{-1}(\Phi(\mathfrak{p}^{+})\setminus \Phi_x)=\tiny{ \begin{tikzpicture}[scale=0.9,baseline=-13pt]
		\hobox{0}{0}{e_1-e_2}
		\hobox{1}{0}{e_{1}-e_3}
            \hobox{2}{0}{\cdots}
             \hobox{3}{0}{e_1-e_{m+1}}
  		\end{tikzpicture}}.$$
Thus we have $\lambda=-\langle \Phi_{\tilde{x}}\rangle=(-m,1^m,0^{n-m-1}).$ Note that $1\leq m\leq n-1$.
The corresponding two root systems are $\rQ(\lambda)=\rR(\lambda)=\mathfrak{su}(1,m)$. 
Recall that for the root system of type $\mathsf{B}_n$, the highest root is $\beta=e_1+e_2$.
Therefore if we write $\lambda=\lambda_0+z\zeta$ with $\zeta$ orthogonal to $\Phi(\mathfrak{k})$, and  $\lambda_{0} \in \mathfrak{h}^{*} $ such that ($\lambda_{0}  + \rho, \beta^\vee  $)=0 and $z=(\lambda+\rho,\beta^\vee) \in \mathbb{R}$, then we will have $z=(\lambda+\rho,\beta^{\vee})=-m+1+2n-2\geq m+1> h^{\vee}_{\rQ(\lambda)}-1=m+1-1=m$. By Theorem \ref{last-unitary}, it follows that $L(\lambda)$ is not unitarizable.

Next we consider the case $\Sp(2n, \mathbb{R})$. In this case, $\Gamma=(\mathsf{C}_n, \mathsf{A}_{n-1})$. By Theorem \ref{BC-Kostan}, there is a bijection between the Kostant modules and the connected subdiagrams
of $\widehat{\Gamma}=(\mathsf{D}_{n+1}, \mathsf{A}_{n})$ containing $\alpha'=\alpha_{n+1}=e_n+e_{n+1}$. The number of these subdiagrams is
$2n$.
By Proposition \ref{P:smooth},  there is a bijection between the smooth Schubert varieties in $G/Q$ and the connected subdiagrams of $\Gamma$
containing the noncompact simple root $\alpha=2e_n$. The number of these subdiagrams is $n+1$.
Thus there are $n-1$ Kostant modules which correspond to rationally smooth (not smooth) Schubert varieties in $G/Q$. By Remark \ref{rat-not-sm}, each of these Kostant modules corresponds to a subdiagram of $\widehat{\Gamma}=(\mathsf{D}_{n+1}, \mathsf{A}_n)$ that does not correspond to any subdiagram of $\Gamma=(\mathsf{C}_n, \mathsf{A}_{n-1})$. By inspecting the corresponding generalized Young diagrams, we find that each of these subdiagrams is isomorphic to $\mathfrak{su}(1,p)$, for some $p$ such that $2\leq p\leq n$:
\begin{equation}
    \begin{pspicture}(-.10,-.1)(2.6,.1)
\cnode*(0,0){.07}{d1}
\uput[d](0,0){\tiny $n-p$}
\cnode*(0.5,0){.07}{d2}
%\uput[d](0.5,0){\small $n-p+1$}
\cnode*(1 ,0){.07}{d3}
\uput[d](1.5,-.065){\small $\cdots$}
\cnode*(1.5,0){.07}{d4}
\cnode*(2,0){.07}{d5}
%\cnode*(2.354,0.354){.07}{d6}
\uput[r](2.1,0.2){\tiny $n-2$}
\cnode*(2.354,-0.354){.07}{d7}
\uput[r](2.354,-0.354){\tiny $n$}
\ncline{d1}{d2}
\ncline[linestyle=dotted,dotsep=1.3pt]{d2}{d3}
\ncline{d3}{d4}
\ncline{d4}{d5}
%\ncline{d5}{d6}
\ncline{d5}{d7}
\pscircle[linewidth=.5pt,fillstyle=solid](0,0){.07}
\pscircle[linewidth=.5pt,fillstyle=solid](0.5,0){.07}
\pscircle[linewidth=.5pt,fillstyle=solid](1,0){.07}
\pscircle[linewidth=.5pt,fillstyle=solid](1.5,0){.07}
\pscircle[linewidth=.5pt,fillstyle=solid](2,0){.07}
%\pscircle[linewidth=.5pt,fillstyle=solid](2.354,.354){.07}
\end{pspicture}
\end{equation}
From \cite[\S 3.7]{EHP}, the  generalized Young diagram for the corresponding $x\in {}^\mathfrak{k}W$ is 
\begin{equation*}
    \text{\tiny{ \begin{tikzpicture}[scale=0.8,baseline=-13pt]
		\hobox{0}{0}{n}
		\hobox{1}{0}{n-1}
            \hobox{2}{0}{\cdots}
		\hobox{3}{0}{\small{n-p+1}}
  		\end{tikzpicture}}}.
\end{equation*}
Correspondingly we have 
$$\Phi_x=\tiny{ \begin{tikzpicture}[scale=1.1,baseline=-17pt]
		\hobox{0}{0}{2e_n}
		\hobox{1}{0}{e_{n-1}+e_n}
            \hobox{2}{0}{\cdots}
		\hobox{3}{0}{\small{e_{n-p+1}+e_n}}
  		\end{tikzpicture}}.$$
The corresponding Kostant module is $L(\lambda)$ with highest weight $\lambda=\tilde{x}\rho-\rho$.
By Lemma \ref{L:rho}, we have $\lambda=-\langle \Phi_{\tilde{x}}\rangle$. By Lemma \ref{L:involution}, we have 
$$\Phi_{\tilde{x}}=w_c^{-1}(\Phi(\mathfrak{p}^{+})\setminus \Phi_x)=\tiny{ \begin{tikzpicture}[scale=0.9,baseline=-89pt]
		\hobox{0}{0}{2e_n}
		\hobox{1}{0}{e_{n-1}+e_n}
        \hobox{1}{1}{\ddots}
        \hobox{2}{2}{2e_{p+1}}
            \hobox{2}{0}{\cdots}
             \hobox{3}{0}{\cdots}
             \hobox{4}{0}{\cdots}
             \hobox{5}{0}{\cdots}
              \hobox{6}{0}{e_2+e_n}
		\hobox{7}{0}{\small{e_{1}+e_n}}
         \hobox{7}{1}{\vdots}
          \hobox{6}{1}{\vdots}
         \hobox{7}{2}{e_1+e_{p+1}}
         \hobox{6}{2}{e_2+e_{p+1}}
          \hobox{6}{3}{e_2+e_{p}}
           \hobox{6}{4}{\vdots}
            \hobox{6}{5}{e_2+e_{3}}
             \hobox{6}{6}{2e_2}
             \hobox{5}{5}{2e_3}
              \hobox{4}{4}{\ddots}
               \hobox{3}{3}{2e_p}
  		\end{tikzpicture}}.$$
Thus we have $\lambda=-\langle \Phi_{\tilde{x}}\rangle=(-(n-p), (-n)^{p-1}, (-n-1)^{n-p}).$ 
The corresponding two root systems are $\rQ(\lambda)=\rR(\lambda)=\mathfrak{su}(1,1)$. 
Recall that for the root system of type $\mathsf{C}_n$, the highest root is $\beta=2e_1$. 
Therefore if we write $\lambda=\lambda_0+z\zeta$ with $\zeta$ orthogonal to $\Phi(\mathfrak{k})$, and $\lambda_{0} \in \mathfrak{h}^{*} $ such that $(\lambda_{0}  + \rho, \beta^\vee) = 0$ and $z=(\lambda+\rho,\beta^\vee) \in \mathbb{R}$, then we will have $z=(\lambda+\rho,\beta^{\vee})=p\geq 2> h^{\vee}_{\rQ(\lambda)}-1=2-1=1$. By Theorem \ref{last-unitary}, it follows that $L(\lambda)$ is not unitarizable.
\end{proof}

\begin{prop}\label{w-exp}
    Suppose that $L(-w\rho-\rho)$ is a unitary highest weight module with highest weight $-w\rho-\rho$. Then 
    the set of simple roots in the corresponding subdiagram $\Gamma'$ is the subset
    \[
       \varIota=\mathrm{supp}(w_cw)= \{ \alpha_i \in \Pi \mid s_i:=s_{\alpha_i} \textup{ appears in a reduced expression of } w_cw \}.
    \]    
    
%     there is a  connected Dynkin subdiagram 
% containing the noncompact simple root $\alpha$ corresponding to $\Phi_{w_cw}$. 
\end{prop}	
\begin{proof}
    Note that $-w\rho-\rho=-w_cx\rho-\rho=w_cxw_0\rho-\rho=\tilde{x}\rho-\rho$ for some $x\in {}^{\fk}W$. 
    Thus from the construction of unitary highest weight modules $L(\lambda_{\Phi'})$, we have $L(-w\rho-\rho)=L(\lambda_{\Phi'})=L_{{x}}$ for some root subsystem $\Phi'$. Therefore, the result follows from \cite[Thm.~6.2]{BH:09} and the proof of Proposition~\ref{generalizeddiagram}.
\end{proof}	

From the proof of Proposition \ref{generalizeddiagram}, we also have the following result.

\begin{cor}\label{nonsimply-schu}
Let $G_{\mathbb{R}}$ be a Lie group of Hermitian type whose complexified Lie algebra $\mathfrak{g}$ is non-simply laced. Suppose $L(-w\rho-\rho)$ is a highest weight Harish-Chandra module of $G_{\mathbb{R}}$. 
  Then  $L(-w\rho-\rho)$ is  unitarizable  if and only if the Schubert variety $X_Q(x^{-1})$ in $G/Q$ is smooth, where $x=w_cw$, and where $Q$ is the parabolic subgroup of $G$ such that ${\rm Lie}(Q)=\fq=\fk+\fp^+$.
\end{cor}

\begin{proof}[Proof of Corollary \ref{uni-bij}]
Suppose $L(-w\rho-\rho)$ is a highest weight Harish-Chandra
module.
From Corollary \ref{simply-schu} and Corollary \ref{nonsimply-schu}, we know that $L(-w\rho-\rho)$ is a unitary highest weight module if and only if  the Schubert variety $X_Q(x^{-1})$ in $G/Q$ is smooth, where $x=w_cw$. From Corollary \ref{s-bij}, the Schubert variety $X_Q(x^{-1})$ in $G/Q$ is smooth if and only if the Schubert variety $X(w_cx)=X(w)$ in $G/B$ is smooth. Therefore $L(-w\rho-\rho)$ is a unitary highest weight module if and only if
the Schubert variety $X(w)$ in $G/B$ is smooth.
\end{proof}
% \begin{rem}
%     From Theorem \ref{thm:bij}, Proposition \ref{P:smooth} and  Corollary \ref{s-bij}, we can see that Corollary \ref{uni-bij} holds.
% \end{rem}

\begin{ex}
    Let $\Gamma=(\mathsf{C}_3, \mathsf{A}_2)$ with the noncompact simple root $\alpha=e_3$. Then we have  $\widehat{\Gamma}=(\mathsf{D}_{4}, \mathsf{A}_{3})$ with $\alpha'=e_3+e_{4}$.
In this case there are $8$ Schubert varieties,
of which $6$ are rationally smooth. Four of the rationally smooth Schubert varieties are in fact smooth; see Figure~\ref{fig:schubert_varieties}. The solid circled nodes correspond to smooth Schubert varieties, and the dotted-circled nodes correspond to rationally smooth but not smooth Schubert varieties.

\begin{figure}[H]
\centering
\begin{minipage}[c]{0.45\textwidth}
    \centering
    % 创建第一个图片内容 (Schubert varieties)
   \Yboxdim{8pt}
\begin{pspicture}(-2,-2)(2,7.5)
\psset{linewidth=.5pt,labelsep=8pt,nodesep=0pt}
%\small
$
\cnode*(0,6.5){.07}{a1}\uput[r](0,6.5){{\ytableausetup{smalltableaux}
                    \begin{ytableau}
                        3 & 2 & 1 \\
                        \none &3 &2 \\
                        \none & \none & 3
                    \end{ytableau}}}
\pscircle(0,6.5){.15}
\cnode*(0,5){.07}{a2} \uput[r](0,5){{\ytableausetup{smalltableaux}
                    \begin{ytableau}
                        3 & 2 & 1 \\
                        \none &3 &2 
                    \end{ytableau}}}
\cnode*(0,3.5){.07}{a3} \uput[r](0,3.8){{\ytableausetup{smalltableaux}
                    \begin{ytableau}
                        3 & 2 & 1 \\
                        \none &3 &\none 
                    \end{ytableau}}}
\cnode*(-1,2.5){.07}{a4} \uput[l](-1,2.5){{\ytableausetup{smalltableaux}
                    \begin{ytableau}
                        3 & 2  \\
                        \none &3
                    \end{ytableau}}}
\pscircle(-1,2.5){.15}
\cnode*(1,2.5){.07}{a5} \uput[r](1,2.5){\scriptsize\young(321)}
\pscircle[linestyle=dotted,dotsep=1.3pt](1,2.5){.15}
\cnode*(0,1.5){.07}{a6} \uput[r](0,1.3){\scriptsize\young(32)}
\pscircle[linestyle=dotted,dotsep=1.3pt](0,1.5){.15}
\cnode*(0,0){.07}{a7} \uput[r](0,0){\scriptsize\young(3)}
\pscircle(0,0){.15}
\cnode*(0,-1.5){.07}{a8} \uput[r](0,-1.5){e}
\pscircle(0,-1.5){.15}
\ncline{a1}{a2}
\ncline{a2}{a3}
\ncline{a3}{a4}
\ncline{a3}{a5}
\ncline{a4}{a6}
\ncline{a5}{a6}
\ncline{a6}{a7}
\ncline{a7}{a8}
\cnode*(2,6.5){.07}{d1}
\cnode*(2.5,6.5){.07}{d2}
\cnode*(3,6.5){.07}{d3}
\ncline{d1}{d2}
\psline[linewidth=.5pt](2.5,6.56)(3,6.56)
\psline[linewidth=.5pt](2.5,6.44)(3,6.44)
\uput[r](2.35,6.5){<}
\pscircle[linewidth=.5pt,fillstyle=solid](2,6.5){.07}
\pscircle[linewidth=.5pt,fillstyle=solid](2.5,6.5){.07}
\small
\uput[d](2,6.5){1}
\uput[d](2.5,6.5){2}
\uput[d](3,6.5){3}
\normalsize
%%%
\cnode*(-3,2.5){.07}{d4}
\cnode*(-2.5 ,2.5){.07}{d5}
\psline[linewidth=.5pt](-3,2.56)(-2.5,2.56)
\psline[linewidth=.5pt](-3,2.44)(-2.5,2.44)
\uput[r](-3.15,2.5){<}
\pscircle[linewidth=.5pt,fillstyle=solid](-3,2.5){.07}
%\pscircle[linewidth=.5pt,fillstyle=solid](-3,2.5){.07}
\small
\uput[d](-3,2.5){2}
\uput[d](-2.5,2.5){3}
%%%
\cnode*(1.5 ,0){.07}{d6}
\uput[d](1.5 ,0){3}
\normalsize
\uput[d](1.5 ,-1.1){\varnothing}
$
\end{pspicture}
%\subcaption{Smooth and rationally smooth Schubert varieties for $(\mathsf{C}_3,\mathsf{A}_2)$}\label{c3}
\end{minipage}
\hfill
\begin{minipage}[c]{0.45\textwidth}
    \centering
    \Yboxdim{8pt}
\begin{pspicture}(-2,-2)(2,7.5)
\psset{linewidth=.5pt,labelsep=8pt,nodesep=0pt}
%\small
$
\cnode*(0,6.5){.07}{a1}\uput[r](0,6.5){{\ytableausetup{smalltableaux}
                    \begin{ytableau}
                        4 & 2 & 1 \\
                        \none &3 &2 \\
                        \none & \none & 4
                    \end{ytableau}}}
\pscircle(0,6.5){.15}
\cnode*(0,5){.07}{a2} \uput[r](0,5){{\ytableausetup{smalltableaux}
                    \begin{ytableau}
                        4 & 2 & 1 \\
                        \none &3 &2
                    \end{ytableau}}}
\cnode*(0,3.5){.07}{a3} \uput[r](0,3.8){{\ytableausetup{smalltableaux}
                    \begin{ytableau}
                        4 & 2 & 1 \\
                        \none &3 &\none
                    \end{ytableau}}}
\cnode*(-1,2.5){.07}{a4} \uput[l](-1,2.5){{\ytableausetup{smalltableaux}
                    \begin{ytableau}
                        4 & 2\\
                        \none &3
                    \end{ytableau}}}
\pscircle(-1,2.5){.15}
\cnode*(1,2.5){.07}{a5} \uput[r](1,2.5){\scriptsize\young(421)}
\pscircle(1,2.5){.15}
\cnode*(0,1.5){.07}{a6} \uput[r](0,1.3){\scriptsize\young(42)}
\pscircle(0,1.5){.15}
\cnode*(0,0){.07}{a7} \uput[r](0,0){\scriptsize\young(4)}
\pscircle(0,0){.15}
\cnode*(0,-1.5){.07}{a8} \uput[r](0,-1.5){e}
\pscircle(0,-1.5){.15}
\ncline{a1}{a2}
\ncline{a2}{a3}
\ncline{a3}{a4}
\ncline{a3}{a5}
\ncline{a4}{a6}
\ncline{a5}{a6}
\ncline{a6}{a7}
\ncline{a7}{a8}
\cnode*(2,6.5){.07}{d1}
\cnode*(2.5,6.5){.07}{d2}
\cnode*(2.8,6.8){.07}{d3}
\cnode*(2.8,6.2){.07}{c1}
\ncline{d1}{d2}
\ncline{d2}{d3}
\ncline{d2}{c1}
% \psline[linewidth=.5pt](2.5,6.56)(3,6.56)
% \psline[linewidth=.5pt](2.5,6.44)(3,6.44)
%\uput[r](2.35,6.5){<}
\pscircle[linewidth=.5pt,fillstyle=solid](2.8,6.8){.07}
\pscircle[linewidth=.5pt,fillstyle=solid](2,6.5){.07}
\pscircle[linewidth=.5pt,fillstyle=solid](2.5,6.5){.07}
\small
\uput[d](2,6.5){1}
\uput[d](2.5,6.5){2}
\uput[d](3.1,6.5){4}
\uput[d](3.1,7.2){3}
\normalsize
\cnode*(-3,2.5){.07}{d5}
\cnode*(-2.5,2.5){.07}{d6}
\cnode*(-3.5 ,2.5){.07}{d4}
\ncline{d4}{d5}
\ncline{d5}{d6}
\pscircle[linewidth=.5pt,fillstyle=solid](-3,2.5){.07}
\pscircle[linewidth=.5pt,fillstyle=solid](-3.5,2.5){.07}
\small
\uput[d](-3,2.5){2}
\uput[d](-3.5,2.5){3}
\uput[d](-2.5,2.5){4}
\cnode*(2.5 ,2.5){.07}{c2}
\cnode*(3,2.5){.07}{c3}
\cnode*(3.5 ,2.5){.07}{c4}
\ncline{c2}{c3}
\ncline{c3}{c4}
%\psline[linewidth=.5pt](-3.5,2.56)(-3,2.56)
%%\psline[linewidth=.5pt](-3.5,2.49)(-3,2.49)
%\uput[r](-3.65,2.5){<}
\pscircle[linewidth=.5pt,fillstyle=solid](2.5,2.5){.07}
%\pscircle[linewidth=.5pt,fillstyle=solid](-3,2.5){.07}
%%\psline[linewidth=.5pt](-4,2.49)(-3.53,2.49)
%\uput[r](-3.65,2.5){<}
\pscircle[linewidth=.5pt,fillstyle=solid](3,2.5){.07}
\small
\uput[d](2.5,2.5){1}
\uput[d](3,2.5){2}
\uput[d](3.5,2.5){4}
\cnode*(1.5,1.5){.07}{d8}
\cnode*(2,1.5){.07}{d9}
%\psline[linewidth=.5pt](-3.5,2.56)(-3,2.56)
%\psline[linewidth=.5pt](-3.5,2.49)(-3,2.49)
%\uput[r](-3.65,2.5){<}
\pscircle[linewidth=.5pt,fillstyle=solid](1.5,1.5){.07}
%\pscircle[linewidth=.5pt,fillstyle=solid](-3,2.5){.07}
%\psline[linewidth=.5pt](1.52,1.5)(2,1.5)
%\uput[r](-3.65,2.5){<}
%\pscircle[linewidth=.5pt,fillstyle=solid](2,1.5){.07}
\small
\uput[d](1.5,1.5){2}
\uput[d](2,1.5){4}
\ncline{d8}{d9}
\cnode*(1.5 ,0){.07}{d7}
\uput[d](1.5 ,0){4}
\normalsize
\uput[d](1.5 ,-1.1){\varnothing}
$
\end{pspicture}
%\subcaption{Smooth Schubert varieties for $(\mathsf{D}_4,\mathsf{A}_3)$}\label{d4}
\end{minipage}
\caption{(Rationally) smooth Schubert varieties for  
$(\mathsf{C}_3,\mathsf{A}_2)$ and $(\mathsf{D}_4, \mathsf{A}_3)$. %\textcolor{red}{Should we explain here the circled nodes and dotted-circled nodes?}
}
\label{fig:schubert_varieties}
\end{figure}
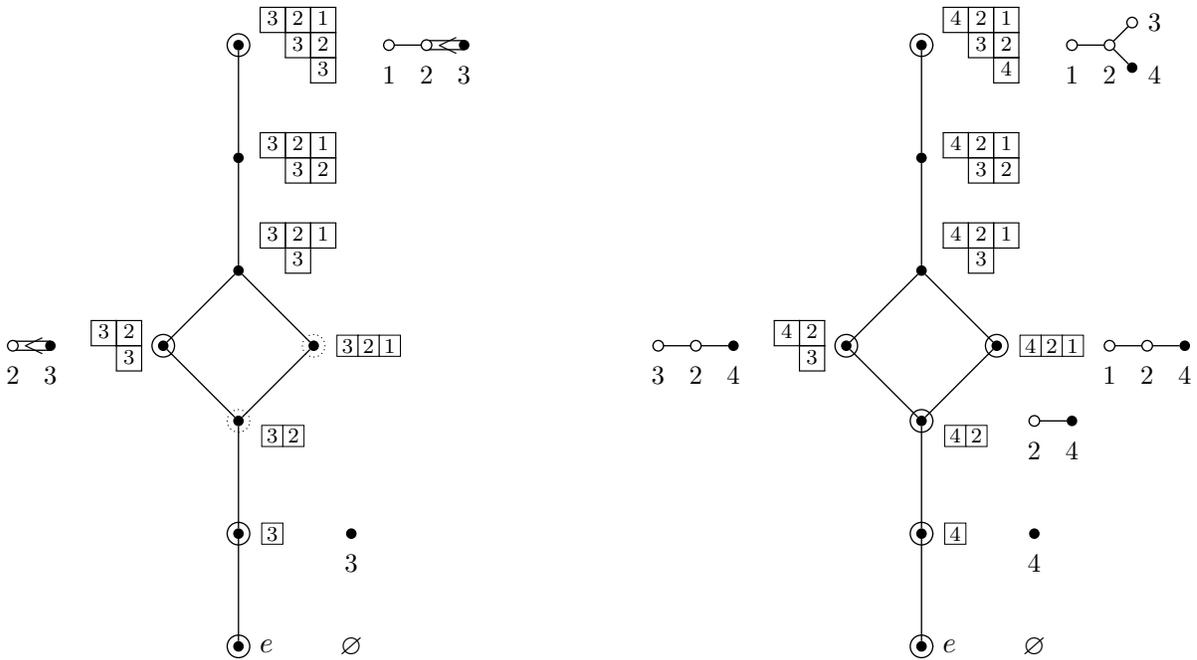

\end{ex}

\section{The number of unitary highest weight modules in a cell}\label{number}   
    
We refer to \cite{KL79}  for the definition of the Kazhdan--Lusztig right cell, and we write $\stackrel{R}{\sim}$ to denote the right cell equivalence relation.
% Recall that $L(w)$ is a highest weight module with the highest weight $\lambda=w\rho-\rho$. 
From Lusztig \cite[Cor. 11.7]{lus03}, we know that $w\stackrel{R}{\sim} x$ if and only if $ww_0\stackrel{R}{\sim}xw_0$.

\begin{thm}[{\cite[Thm.~1.1]{BBLX25}}]\label{one-cell}
	Suppose that $L(-w\rho-\rho)$ and $L(-x\rho-\rho)$ are highest weight Harish-Chandra modules. Then $w\stackrel{R}{\sim} x$ if and only if $\AV(L(-w\rho-\rho))=\AV(L(-x\rho-\rho))=\overline{\mathcal{O}_k}$ for some $0\leq k\leq r$.
\end{thm}

Recall that there are two different ways to compute the associated variety  of a highest weight Harish-Chandra module $L(\lambda)$ or $L(-w\rho-\rho)$; see \cite[Thm.~6.2 and Thm.~7.1]{BXX} and \cite[Thm.~1.2]{BHXZ}.
From the above Theorem~\ref{one-cell}, we can see that there is only one right cell $\mathcal{C}_k$ such that the corresponding highest weight Harish-Chandra modules $L(-w\rho-\rho)$ with $w\in \mathcal{C}_k $ have the same associated variety $\overline{\mathcal{O}_k}$.
From the proof of our Theorem \ref{thm:bij}, we can obtain the following result. 

\begin{prop}\label{card}
    Suppose $\Phi' \subseteq \Phi$ is a root subsystem corresponding to a connected Dynkin subdiagram $\Gamma'$ that contains the noncompact simple root. Define $k(\Gamma') := \frac{h^{\vee}-h'^{\vee}}{c}$. Then we have
    $$\AV(L(\lambda_{\Phi'}))=\begin{cases} \overline{\mathcal{O}_{k(\Gamma')}} &\textup{if~} 0\leq k(\Gamma')\leq r-1  \\
\mathfrak{p}^+ &\textup{if~} k(\Gamma')>r-1.
         \end{cases}$$  
\end{prop}
\begin{proof}
    From Proposition \ref{C: dimYk}, we put $k=k(\lambda_{\Phi'})=-\frac{(\lambda_{\Phi'},\beta^{\vee})}{c}$. From equation (\ref{eq-lambda}) and (\ref{zeprime}), we have $k=k(\lambda_{\Phi'})=\frac{h^{\vee}-h'^{\vee}}{c}=k(\Gamma')$.
    Thus the result follows from Proposition \ref{C: dimYk}. 
\end{proof}

From Proposition \ref{card} and the arguments in the proof of Theorem \ref{thm:bij}, we have the following result.

\begin{cor}\label{numberNk}
Denote the number of unitary highest weight modules in a given right cell $\mathcal{C}_k$ by 
\begin{align*}
N_k
:=& \: \#\{w\in \mathcal{W} \mid L(-w\rho-\rho) \textup{~is~unitarizable~and~} w \in \mathcal{C}_k\}\\
=& \:\#\{x\in {}^\mathfrak{k}W \mid L_x \textup{~is~unitarizable~and~} x \in w_c\mathcal{C}_k\}\\
=& \: \#\{\tilde{w}\in {}^\mathfrak{k}W \mid L(\tilde{w}\rho-\rho) \textup{~is~unitarizable~and~} \tilde{w} \in \mathcal{C}_kw_0\}.\end{align*}
Then we have $$
N_k=\begin{cases} \#\left\{\Gamma'\subset \Gamma \middle|
\renewcommand{\arraystretch}{.75}
\begin{array}{l}
\Gamma'\subset \Gamma\textup{~is~a~connected~subdiagram~containing~}\\
\textup{the~noncompact~simple~root~and } h'^{\vee}=h^{\vee}-kc  
\end{array}
\right\} &\textup{if~} 0\leq k\leq r-1,  \\[20pt]
\#\left\{\Gamma'\subset \Gamma \middle|
\renewcommand{\arraystretch}{.75}
\begin{array}{l}
\Gamma'\subset \Gamma\textup{~is~a~connected~subdiagram~containing~}\\
\textup{the~noncompact~simple~root~and } h'^{\vee}<h^{\vee}-(r-1)c  
\end{array}
\right\} &\textup{if~} k=r.
         \end{cases}
%\#\left\{\Gamma'\subset \Gamma \middle|
% \begin{array}{c}
% \Gamma'\subset \Gamma\text{~is~a~connected~subdiagram~containing~}\\
% \text{the~noncompact~simple~root~and } h'^{\vee}=h^{\vee}-kc  
% \end{array}
% \right\},
$$
Here $h'^{\vee}$ \textup{(}resp., $h^{\vee}$\textup{)} is the dual Coxeter number of $\Gamma'$ \textup{(}resp., $\Gamma$\textup{)}, and the value of the parameter~$c$ is given in Table~\ref{constants-k}.
% Let $L(w) $ be a unitary highest weight module with $V(L(w))=\overline{\mathcal{O}}_k$ for some $0\leq k\leq r$. Then 
\end{cor}

    Recall that $L(\lambda_{\Phi'})=L_x=L(-w_cx\rho-\rho)=L(\tilde{x}\rho-\rho)$ for some $x\in {}^\mathfrak{k}W$. From Proposition \ref{decomp} and Proposition \ref{P:product}, we know that $w_cx$ will be a reduced expression if the reduced expression of $x$ comes from a generalized Young diagram. From Lemma \ref{L:involution}, the reduced expression of $\tilde{x}=w_cxw_0$ can be read off from the generalized Young diagram $\Phi_{\tilde{x}}=w_{c}^{-1} (\Phi(\mathfrak{p}^+)\setminus \Phi_x)=\Phi(\mathfrak{p}^+)\setminus w_c(\Phi_x)$. See the arguments in the proof of Proposition \ref{generalizeddiagram}.

\begin{cor}
The highest weight module $L(\lambda)$ with trivial infinitesimal character is unitarizable if and only if $\lambda=\lambda_{\Phi'}$ for some root
subsystem $\Phi'$ corresponding to a connected Dynkin subdiagram $\Gamma'$ of $\Gamma$ that contains the noncompact simple root.
    The highest weight module $L(-w\rho-\rho)$ is unitarizable if and only if $w=w_cx$ for some $x\in {}^\mathfrak{k}W$ corresponding to some connected subdiagram of $\Gamma$ that contains 
the noncompact simple root. The explicit expressions for $\lambda_{\Phi'}$, $w_c$, and $x$ are given in the following subsections.
\end{cor}

Note that $\Gamma'=\Gamma$ if and only if $h'=h$, if and only if $k=0$. In this case, the corresponding element $x=w^{c}$ is the longest element in ${}^\mathfrak{k}W$. Thus we have $w_cx=w_cw^c=w_0\in \mathcal{W}$ and $w(\Gamma')=w_cw^cw_0=w_0^2=\mathrm{Id}$. In this case, we have $L(\lambda)=L(0)$.

Similarly, $\Gamma'=\varnothing$ if and only if   $h'=0$, which implies that $k=r$. In this case, we have the corresponding element $x=\rm{Id}$. Thus we have $w_cx=w_c\in \mathcal{W}$ and $w(\Gamma')=w_cxw_0=w_cw_0$. In this case, we have  
$L(\lambda)=L(-w_c\rho-\rho)=N(-w_c\rho-\rho)$ by \cite[Thm.~9.12]{hum08}, since $w_c(\Phi(\mathfrak{p}^+))\subset \Phi(\mathfrak{p}^+)$. Note that $-w_c\rho-\rho=-w_c(\rho_c+\rho_n)-\rho=\rho_c-\rho_n-(\rho_c+\rho_n)=-2\rho_n=-2(\rho-\rho_c)$, where $\rho_c$ denotes half the sum of roots in $\Phi^+(\mathfrak{k})$ and $\rho_n$ denotes half the sum of roots in $\Phi(\mathfrak{p}^+)$.

\begin{table}[H] 
\centering
\begingroup
\renewcommand{\arraystretch}{1.5}
\setlength\tabcolsep{10pt}
%\begin{tabular}{lll}
\begin{tabular}{l|l|l}
\hline
$\fg_\R$ & $k$ & $N_k$\\  
\hline    
$\mathfrak{su}(p,q)$ &
$[1,\min\{p,q\}-1]$& $k+1$\\ 
&$\min\{p,q\}$& $\frac{p(2q-p-1)}{2}+1$ \\ \hline
   $\mathfrak{sp}(n,\R)$ & odd and $[1,n-1]$ & $0$\\
   & even and $[1,n-1]$ & $1$\\
   & $n$ & $[\frac{n}{2}]+1$\\ \hline
   $\mathfrak{so}^*(2n)$  &$[1,[\frac{n}{2}]-2]$ & $1$\\
   & $[\frac{n}{2}]-1$ and $ n$ odd  & $1$\\
   & $[\frac{n}{2}]-1$ and $ n$ even & $2$\\
   & $[\frac{n}{2}]$ & $3n-3[\frac{n}{2}]-3$\\ \hline
%\midrule  
$\mathfrak{so}(2,2n-1)$ &  $1$ & $0$\\
&$2$& $n$\\ \hline
$\mathfrak{so}(2,2n-2)$ &   $1$ & $2$\\
&$2$& $n-1$\\ \hline
%\midrule  
 $\mathfrak{e}_{6(-14)}$ &   $1$ & $0$\\
&$2$& $8$\\   \hline
$\mathfrak{e}_{7(-25)}$ &   $1$ & $0$\\
&  $2$  & $1$\\
&  $3$  & $8$\\ \hline
\end{tabular}
\endgroup
\caption{The number of  unitary highest weight modules in a  right cell. (For all types, $N_0=1$.)}
\label{Nk}
\bigskip
\end{table}

In what follows, we will give the results in Table \ref{Nk}. We only consider  the cases $1\leq k\leq r$ and $\Gamma'\neq \Gamma$, $\varnothing$.
Define $S(\Gamma') := \{i\mid \alpha_i\in \Gamma'\}$. We write $w_{X_k}$ to denote the longest element in the Weyl group $W(X_k)
$.

\subsection{Case $\g_\R=\su(p,q)$}From \cite[Table 1]{BKOP} and \cite{EHW}, we have $$w_c=w_{A_{p-1}}\times w_{A_{q-1}}=s_{p-1}(s_{p-2}s_{p-1})\cdots (s_1s_2\cdots s_{p-1})s_n(s_{n-1}s_n)\cdots (s_{p+1}s_{p+2}\cdots s_n),$$
where $n=p+q$.
For  $1\leq k\leq r=\min\{p,q\}$,  we  have $h^{\vee}-kc=p+q-k$ by Table~\ref{constants-k}. Note that $S(\Gamma')=\{i\mid p-p'+1\leq i\leq p+q'-1\}$ and $\Gamma'\simeq \mathfrak{su}(p',q')$ with $1\leq p'\leq p$ and $1\leq q'\leq q$. Then $h'^{\vee}=p'+q'$. 

Suppose $p\leq q$. Then $r=p$. We write $k=r-k'=p-k'$. 

When $k\leq r-1$, equivalently $k'\geq 1$, from Corollary \ref{numberNk} we will have $p'+q'=p+q-k=p+q-(r-k')=q+k'$ since $h'^{\vee}=h^{\vee}-kc$. Thus
$(p',q')=(k',q),(k'+1,q-1),\dots,(p-1,q-p+k'+1)\text{~or~} (p,q-p+k')$. Thus $N_k=p-(k'-1)=p-(p-k-1)=k+1.$

When $k=r$, equivalently $k'=0$, from Corollary \ref{numberNk} we will have  $p'+q'<p+q-r+1=p+q-r+1=q+1$ since $h'^{\vee}<h^{\vee}-(r-1)c$. Denote $$d_r=\#\{(p',q')\in \mathbb{Z}_{>0}\times \mathbb{Z}_{>0}\mid 1\leq p'\leq p, 1\leq q'\leq q, 2\leq p'+q'\leq q\}=\frac{p(2q-p-1)}{2}.$$ 

Therefore, we have 
\begin{align*}N_k&=\begin{cases} k+1, &\text{if~} 1\leq k\leq r-1  \\
d_r+1, &\text{if~} k=r.
         \end{cases}
\end{align*}
From \cite[\S 3.7]{EHP}, the  generalized Young diagram of the corresponding $x\in {}^\mathfrak{k}W$ for $\Gamma'\simeq \mathfrak{su}(p',q')$ is 
\begin{equation*}
    \small{ \begin{tikzpicture}[scale=1.2,baseline=-69pt]
		\hobox{0}{0}{p}
		\hobox{1}{0}{p-1}
            \hobox{2}{0}{\cdots}
		\hobox{3}{0}{p-p'+1}
        \hobox{0}{1}{p+1}
         \hobox{1}{1}{p}
          \hobox{2}{1}{\cdots}
           \hobox{3}{1}{p-p'+2}
        \hobox{0}{2}{\vdots}
         \hobox{1}{2}{\vdots}
          \hobox{2}{2}{\cdots}
           \hobox{3}{2}{\vdots}
        \hobox{0}{3}{p+q'-1}
         \hobox{1}{3}{p+q'-2}
          \hobox{2}{3}{\cdots}
           \hobox{3}{3}{p-p'+q'}
  		\end{tikzpicture}}.
\end{equation*}
Thus we have \[x=s_ps_{p-1}\cdots s_{p-p'+1}s_{p+1}\cdots s_{p-p'+2} s_{p+q'-1}\cdots s_{p-p'+q'}\] and 
\begin{align*}
\lambda_{\Phi'} 
&= -h^\vee\zeta + h'^\vee w_c\zeta' \\
&= (\underbrace{-q,\ldots,-q}_{p};\ \underbrace{p,\ldots,p}_{q}) 
+ w_c\big(\underbrace{0,\ldots,0}_{p-p'},\underbrace{q',\ldots,q'}_{p'};\ \underbrace{-p',\ldots,-p'}_{q'},\underbrace{0,\ldots,0}_{q-q'}\big) \\
&= (\underbrace{-q,\ldots,-q}_{p};\ \underbrace{p,\ldots,p}_{q}) 
+ \big(\underbrace{q',\ldots,q'}_{p'},\underbrace{0,\ldots,0}_{p-p'};\ \underbrace{0,\ldots,0}_{q-q'},\underbrace{-p',\ldots,-p'}_{q'}\big) \\
&= (\underbrace{-(q-q'),\ldots,-(q-q')}_{p'},\underbrace{-q,\ldots,-q}_{p-p'};\ \underbrace{p,\ldots,p}_{q-q'},\underbrace{p-p',\ldots,p-p'}_{q'}) \\
&= q'\omega_{p'} - (p+q)\omega_p + p'\omega_{p+q-q'}.
\end{align*}

  When $\Gamma'=\varnothing$, we have \begin{align*}
      \lambda_{\Phi'}=&-w_c\rho-\rho=-2(\rho-\rho_c)\\
      =&-((n-1,n-3,\dots,-n+3,-n+1)\\
      &-(p-1,p-3,\dots,-p+3,-p+1,q-1,q-3,\dots,-q+3,-q+1))\\
      =&-(\underbrace{q,\ldots,q}_{p},\underbrace{-p,\ldots,-p}_{q})\\
      =&(\underbrace{-q,\ldots,-q}_{p},\underbrace{p,\ldots,p}_{q}).
  \end{align*}
   
 \subsection{Case $\g_\R=\so(2,2n-1)$}From \cite[Table 1]{BKOP} and \cite{EHW}, we have $$w_c=w_{B_{n-1}}=s_n(s_{n-1}s_ns_{n-1})\cdots (s_2\cdots s_{n-1}s_n s_{n-1}\cdots s_2).$$
For  $1\leq k \leq r=2$, we  have $h^{\vee}-kc=2n-1-k(n-\frac{3}{2})$ by Table \ref{constants-k}.   Note that  $S(\Gamma')=\{i\mid 1\leq i\leq p\}$ and $\Gamma'\simeq \mathfrak{su}(1,p)$ with $1\leq p\leq n-1$. Then $h'^{\vee}=p+1$. 

When $k=1$,  from Corollary \ref{numberNk} we will have $p+1=2n-1-(n-\frac{3}{2})$ since $h'^{\vee}=h^{\vee}-kc$, which implies that $p=n-\frac{1}{2}$. This is a contradiction since $p$ is a positive integer.
Thus $N_1=0$.

When $k=r=2$, from Corollary \ref{numberNk} we will have $p+1<2n-1-(n-\frac{3}{2})$  since $h'^{\vee}<h^{\vee}-(r-1)c$, which simplifies to $p<n-\frac{1}{2}$, equivalently, $1\leq p\leq n-1$.  Thus $N_1=n-1+1=n$.

Therefore, we have 
\begin{align*}N_k&=\begin{cases} 0, &\text{if~} k=1  \\
n, &\text{if~} k=2.
         \end{cases}
\end{align*}
    
From \cite[\S 3.7]{EHP}, the  generalized Young diagram of the corresponding $x\in {}^\mathfrak{k}W$ for $\Gamma'\simeq \mathfrak{su}(1,p)$ with $1\leq p\leq n-1$ is 
\begin{equation*}
    { \begin{tikzpicture}[scale=0.7,baseline=-10pt]
		\hobox{0}{0}{1}
		\hobox{1}{0}{2}
            \hobox{2}{0}{\cdots}
		\hobox{3}{0}{p}
  		\end{tikzpicture}}.
\end{equation*}
Thus we have $$x=s_{1}s_{2}\cdots s_{p}$$ and 
$$\lambda_{\Phi'}=-(2n-p) \omega_1+\omega_{p+1}.$$
When $\Gamma'=\varnothing$, we have \begin{align*}
      \lambda_{\Phi'}=&-w_c\rho-\rho=-2(\rho-\rho_c)\\
      =&-((2n-1,2n-3,\dots,3,1)-(0,2n-3,2n-5,\dots,3,1))\\
      =&-(2n-1,0,\dots,0).
  \end{align*}

\subsection{Case $\g_\R=\sp(n, \RR)$}From \cite[Table 1]{BKOP} and \cite{EHW}, we have $$w_c=w_{A_{n-1}}=s_{n-1}(s_{n-2}s_{n-1})\cdots (s_1s_2\cdots s_{n-1}).$$
For  $1\leq k \leq r=n$,    we  have $h^{\vee}-kc=n+1-\frac{k}{2}$ by Table \ref{constants-k}.     Note that  $S(\Gamma')=\{i\mid n-p+1\leq i\leq n\}$ and $\Gamma'\simeq \mathfrak{sp}(p, \RR)$ with $1\leq p\leq n-1$. Then $h'^{\vee}=p+1$.

  When $k\leq r-1=n-1$,  from Corollary \ref{numberNk} we will have $p+1=n+1-\frac{k}{2}$ since $h'^{\vee}=h^{\vee}-kc$, which simplifies to $p=n-\frac{k}{2}$. Thus $N_k=1$ when $k$ is even and $N_k=0$ when $k$ is odd.

When $k=r$,  from Corollary \ref{numberNk} we will have $p+1<n+1-(n-1)\frac{1}{2}$ since $h'^{\vee}<h^{\vee}-(r-1)c$, which simplifies to $p<\frac{n+1}{2}$, equivalently $p\leq [\frac{n}{2}]$. 

Therefore, we have 
\begin{align*}N_k&=\begin{cases} 0 &\text{if~$k$~is~odd~and~} 1\leq k\leq n-1,  \\
1 &\text{if~$k$~is~even~and~} 1\leq k\leq n-1,  \\
[\frac{n}{2}]+1 &\text{if~} k=n.
         \end{cases}
\end{align*}
    
From \cite[\S 3.7]{EHP}, the  generalized Young diagram of the corresponding $x\in {}^\mathfrak{k}W$ for $\Gamma'\simeq\mathfrak{sp}(p, \RR)$ with $1\leq p\leq n$ is 
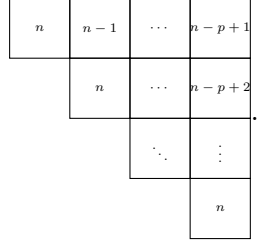
\begin{equation*}
    \tiny{ \begin{tikzpicture}[scale=0.8,baseline=-46pt]
		\hobox{0}{0}{n}
		\hobox{1}{0}{n-1}
            \hobox{2}{0}{\cdots}
		\hobox{3}{0}{n-p+1}
         \hobox{1}{1}{n}
          \hobox{2}{1}{\cdots}
           \hobox{3}{1}{n-p+2}
          \hobox{2}{2}{\ddots}
           \hobox{3}{2}{\vdots}
    \hobox{3}{3}{n}
  		\end{tikzpicture}}.
\end{equation*}
Thus we have $$x=s_ns_{n-1}\cdots s_{n-p+1}s_{n}\cdots s_{n-p+2}\cdots s_{n}$$ and \begin{align*}
\lambda_{\Phi'} 
&= -h^\vee\zeta + h'^\vee w_c\zeta' \\
&= (\underbrace{-(n+1),\ldots, -(n+1)}_{n}) 
+ w_c\big(\underbrace{0,\ldots,0}_{n - p},\ \underbrace{p+1,\ldots,p+1}_{p}\big) \\
&= (\underbrace{-(n+1),\ldots, -(n+1)}_{n}) 
+ \big(\underbrace{p+1,\ldots,p+1}_{p},\ \underbrace{0,\ldots,0}_{n - p}\big) \\
&= (\underbrace{-(n - p),\ldots, -(n - p)}_{p},\ \underbrace{-(n+1),\ldots, -(n+1)}_{n - p}) \\
&= (p+1)\omega_{p} - (n+1)\omega_n.
\end{align*}  
    
   When $\Gamma'=\varnothing$, we have \begin{align*}
      \lambda_{\Phi'}=&-w_c\rho-\rho=-2(\rho-\rho_c)\\
      =&-((2n,2n-2,\dots,4,2)-(n-1,n-3,\dots,-n+3,-n+1))\\
      =&-(n+1,\dots,n+1).
  \end{align*}   
    
  \subsection{Case \texorpdfstring{$\g_\R=\so^*(2n)$}{gR=so*(2n)}}
  From \cite[Table 1]{BKOP} and \cite{EHW}, we have $$w_c=w_{A_{n-1}}=s_{n-1}(s_{n-2}s_{n-1})\cdots (s_1s_2\cdots s_{n-1}).$$
For  $1\leq k \leq r=[\frac{n}{2}]$ with $n\geq 4$,    we  have $h^{\vee}-kc=2n-2-2k$ by Table \ref{constants-k}.   Note that  $S(\Gamma')=\{i\mid n-p\leq i\leq n-2 \text{~or~} i=n\}$ and $\Gamma'\simeq \mathfrak{su}(1, p)$ with $1\leq p\leq n-1$, or $S(\Gamma')=\{i\mid n-q+1\leq i\leq n\}$ and $\Gamma'\simeq\so^*(2q)$ with $3\leq q\leq n-1$. Then $h'^{\vee}=p+1$ or $2q-2$.

  When $k\leq r-1$ and $\Gamma'\simeq \mathfrak{su}(1,p)$ with $1\leq p\leq n-1$,  from Corollary \ref{numberNk} we will have $p+1=2n-2-2k$ since $h'^{\vee}=h^{\vee}-kc$, which simplifies to $p=2n-3-2k\leq n-1$. Thus we get \begin{equation}\label{unique}
  \frac{n}{2}-1\leq k\leq \left[\frac{n}{2}\right]-1.\end{equation}
  When $n$ is even, there is a unique $k$ satisfying equation (\ref{unique}), i.e., $k=r-1=\frac{n}{2}-1$. In this case, we have $p=n-1$.  When $n$ is odd, we get a contradiction. 
  When $\Gamma'\simeq \so^*(2q)$ with $3\leq q\leq n$, from Corollary \ref{numberNk} we will have $2q-2=2n-2-2k$ since $h'^{\vee}=h^{\vee}-kc$, which simplifies to $q=n-k$.

When $k=r$ and $\Gamma'\simeq \mathfrak{su}(1,p)$ with $1\leq p\leq n-1$,  from Corollary \ref{numberNk} we will have $p+1<2n-2-2(r-1)$ since $h'^{\vee}<h^{\vee}-(r-1)c$, which implies that $p<2n-1-2r$. 
  When $\Gamma'\simeq \so^*(2q)$ with $3\leq q\leq n-1$, from Corollary \ref{numberNk} we will have $2q-2<2n-2-2(r-1)$ since $h'^{\vee}<h^{\vee}-(r-1)c$, which implies that $3\leq q<n-r+1$.  Thus $N_r=1+2n-2-2r+n-r-2=3n-3r-3$.

Therefore, we have 
\begin{align*}N_k&=\begin{cases} 
1, &\text{if~} 1\leq k\leq [\frac{n}{2}]-2  \\
1, &\text{if~$n$~is~odd~and~} k= [\frac{n}{2}]-1  \\
2, &\text{if~$n$~is~even~and~} k= [\frac{n}{2}]-1  \\
3n-3r-3, &\text{if~} k= [\frac{n}{2}].
         \end{cases}
\end{align*}
    
From \cite[\S 3.7]{EHP}, the  generalized Young diagram of the corresponding $x\in {}^\mathfrak{k}W$ for $\Gamma'\simeq \mathfrak{su}(1,p)$ with $2\leq p\leq n-1$ is 
\begin{equation*}
    \small{ \begin{tikzpicture}[scale=0.7,baseline=-12pt]
		\hobox{0}{0}{n}
		\hobox{1}{0}{n-2}
            \hobox{2}{0}{n-3}
             \hobox{3}{0}{\cdots}
		\hobox{4}{0}{n-p}
  		\end{tikzpicture}}.
\end{equation*}
Thus we have $$x=s_ns_{n-2}\cdots s_{n-p}.$$ Note that for $\Gamma'\simeq \mathfrak{su}(1,1)$, we have $x=s_n$. Also for $1\leq p\leq n-2$, we have \begin{align*}
\lambda_{\Phi'}&= -h^\vee\zeta + h'^\vee w_c\zeta'\\
&=(\underbrace{-(n-1),\ldots, -(n-1)}_{n})+w_c(\underbrace{0,\ldots,0}_{n-p-1},\underbrace{1,\ldots ,1}_{p},p))\\
&=(\underbrace{-(n-1),\ldots, -(n-1)}_{n})+(p,\underbrace{1,\ldots ,1}_{p},\underbrace{0,\ldots,0}_{n-p-1})\\
&=(-(n-p-1),\underbrace{-(n-2),\ldots, -(n-2)}_{p},\underbrace{-(n-1),\ldots, -(n-1)}_{n-p-1})\\
&=(p-1) \omega_1+\omega_{p+1}-(2n-2)\omega_n.
\end{align*}    
 When $p=n-1$, similarly we have $\lambda_{\Phi'}=(n-2) \omega_1-(2n-4)\omega_n.$
 
   From \cite[\S 3.7]{EHP}, the  generalized Young diagram of the corresponding $x\in {}^\mathfrak{k}W$ for $\Gamma'\simeq \mathfrak{so}^{*}(2q)$ with $3\leq q\leq n-1$ is 
\begin{equation*}
    \tiny{ \begin{tikzpicture}[scale=0.8,baseline=-58pt]
		\hobox{0}{0}{n}
		\hobox{1}{0}{n-2}
            \hobox{2}{0}{n-3}
             \hobox{3}{0}{\cdots}
		\hobox{4}{0}{n-q+1}
         \hobox{1}{1}{n-1}
          \hobox{2}{1}{n-2}
           \hobox{3}{1}{\cdots}
            \hobox{4}{1}{n-q+2}
          \hobox{2}{2}{n}
           \hobox{3}{2}{\cdots}
           \hobox{4}{2}{n-q+3}
    \hobox{3}{3}{\ddots}
     \hobox{4}{3}{\vdots}
      \hobox{4}{4}{n'}
  		\end{tikzpicture}},
\end{equation*}
where $n'=n$ when $q$ is even and $n'=n-1$ when $q$ is odd.

Thus we have $$x=s_ns_{n-2}\cdots s_{n-q+1}s_{n-1}\cdots s_{n-q+2}\cdots s_{n'}$$ and \begin{align*}
\lambda_{\Phi'}&= -h^\vee\zeta + h'^\vee w_c\zeta'\\
&=(\underbrace{-(n-1),\ldots, -(n-1)}_{n})+w_c(\underbrace{0,\ldots,0}_{n-q},\underbrace{q-1,\ldots ,q-1}_{q})\\
&=(\underbrace{-(n-1),\ldots, -(n-1)}_{n})+(\underbrace{q-1,\ldots ,q-1}_{q},\underbrace{0,\ldots,0}_{n-q})\\
&=(\underbrace{-(n-q),\ldots, -(n-q)}_{q},\underbrace{-(n-1),\ldots, -(n-1)}_{n-q})\\
&=(q-1) \omega_{q}-(2n-2)\omega_n.
\end{align*}     
     
  When $\Gamma'=\varnothing$, we have \begin{align*}
      \lambda_{\Phi'}=&-w_c\rho-\rho=-2(\rho-\rho_c)\\
      =&-((2n-2,2n-4,\dots,2,0)-(n-1,n-3,\dots,-n+3,-n+1))\\
      =&-(n-1,\dots,n-1).
  \end{align*}   
    
  \subsection{Case $\g_\R=\so(2,2n-2)$}
  From \cite[Table 1]{BKOP} and \cite{EHW}, we have \begin{align*}w_c=w_{D_{n-1}}=&
      s_ns_{n-1}(s_{n-2}s_ns_{n-1}s_{n-2})\cdots(s_k\cdots s_{n-2}s_ns_{n-1}\cdots s_k)\cdots(s_2\cdots s_{n-2}s_n s_{n-1}\cdots s_{2}).
  \end{align*}
From Bourbaki \cite[PLATE IV]{Bour}, we know that $w_c(x_1,\dots,x_n)=(x_1,-x_2,\dots,-x_n)$ if $n$ is odd and $w_c(x_1,\dots,x_n)=(x_1,-x_2,\dots,-x_{n-1},x_n)$ if $n$ is even for any vector $(x_1,\dots,x_n)\in \mathbb{R}^n$.

For  $1\leq k \leq r=2$ with $n\geq 4$,    we  have $h^{\vee}-kc=2n-2-k(n-2)$ by Table \ref{constants-k}.  Note that $\Gamma'\simeq  \mathfrak{su}(1, p)$ with $1\leq p\leq n-1$. When $1\leq p\leq n-2$, we have $S(\Gamma')=\{i\mid 1\leq i\leq p\}$ and $\Gamma'\simeq  \mathfrak{su}(1, p)$. When  $\Gamma'\simeq  \mathfrak{su}(1, n-1)$,  there are two different subdiagrams which are isomorphic to it. We denote them by $\mathfrak{su}(1, n-1)^I$ (with $S(\Gamma')=\{i\mid 1\leq i\leq n-1\}$) and $\mathfrak{su}(1, n-1)^{II}$ (with $S(\Gamma')=\{i\mid 1\leq i\leq n-2 \text{~or~} i=n\}$). Then $h'^{\vee}=p+1$.

  When $k=1$,  from Corollary \ref{numberNk} we will have $p+1=2n-2-(n-2)$ since $h'^{\vee}=h^{\vee}-kc$, which simplifies to $p=n-1$.

When $k=r=2$,  from Corollary \ref{numberNk} we will have $p+1<2n-2-c$ since $h'^{\vee}<h^{\vee}-(r-1)c$, which simplifies to $p<n-1$.

Therefore, we have 
\begin{align*}N_k&=\begin{cases} 
2, &\text{if~} k=1  \\
n-1, &\text{if~} k=2.
         \end{cases}
\end{align*}
    
From \cite[\S 3.7]{EHP}, the  generalized Young diagram of the corresponding $x\in {}^\mathfrak{k}W$ for $\Gamma'\simeq\mathfrak{su}(1,p)$ with $1\leq p\leq n-2$ is 
\begin{equation*}
    \text{{ \begin{tikzpicture}[scale=0.7,baseline=-12pt]
		\hobox{0}{0}{1}
		\hobox{1}{0}{2}
            \hobox{2}{0}{\cdots}
		\hobox{3}{0}{p}
  		\end{tikzpicture}}}.
\end{equation*}
Thus we have $$x=s_1s_{2}\cdots s_{p}.$$  
For $1\leq p\leq n-3$, similarly we have $$\lambda_{\Phi'}=-(2n-p-1)\omega_1+\omega_{p+1}.$$ For $p=n-2$, similarly we have $$\lambda_{\Phi'}=-(n+1)\omega_1+\omega_{n-1}+\omega_n.$$

  The  generalized Young diagram of the corresponding $x\in {}^\mathfrak{k}W$ for $\Gamma'\simeq\mathfrak{su}(1,n-1)^I$  is 
  \begin{equation*}
    \text{{ \begin{tikzpicture}[scale=0.7,baseline=-12pt]
		\hobox{0}{0}{1}
		\hobox{1}{0}{2}
            \hobox{2}{0}{\cdots}
		\hobox{3}{0}{n-1}
  		\end{tikzpicture}}}.
\end{equation*}
Thus we have $$x=s_1s_{2}\cdots s_{n-1}.$$ 
When $n$ is even, we have
\begin{align*}
\lambda_{\Phi'}&= -h^\vee\zeta + h'^\vee w_c\zeta'\\
&=(-2n+2,\underbrace{0,\ldots, 0)}_{n-1})+w_c(n-1,\underbrace{-1,\ldots,-1}_{n-1})\\
&=(-2n+2,\underbrace{0,\ldots, 0)}_{n-1})+(n-1,\underbrace{1,\ldots,1}_{n-2},-1)\\
&=(-n+1,\underbrace{1,\ldots,1}_{n-2},-1)\\
&=-n \omega_{1}+2\omega_{n-1}.
\end{align*}     

When $n$ is odd, the computation is similar. Thus we have 
 $$\lambda_{\Phi'}=\begin{cases} 
-n\omega_1+2\omega_{n-1}, &\text{if~} n \text{~is~even} \\
-n\omega_1+2\omega_{n}, &\text{if~} n \text{~is~odd}.
         \end{cases}
$$

     The  generalized Young diagram of the corresponding $x\in {}^\mathfrak{k}W$ for $\Gamma'\simeq\mathfrak{su}(1,n-1)^{II}$  is 
  \begin{equation*}
    \text{{ \begin{tikzpicture}[scale=0.7,baseline=-20pt]
		\hobox{0}{0}{1}
		\hobox{1}{0}{2}
            \hobox{2}{0}{\cdots}
		\hobox{3}{0}{n-2}
        \hobox{3}{1}{n}
  		\end{tikzpicture}}}.
\end{equation*}
  Thus we have $$x=s_1s_{2}\cdots s_{n-2}s_{n}$$ and similarly we can get $$\lambda_{\Phi'}=\begin{cases} 
-n\omega_1+2\omega_{n}, &\text{if~} n \text{~is~even} \\
-n\omega_1+2\omega_{n-1}, &\text{if~} n \text{~is~odd}.
         \end{cases}
$$
    
      When $\Gamma'=\varnothing$, we have \begin{align*}
      \lambda_{\Phi'}=&-w_c\rho-\rho=-2(\rho-\rho_c)\\
      =&-((2n-2,2n-4,\dots,2,0)-(0,2n-4,\dots,2,0))\\
      =&-(2n-2,0,\cdots,0).
  \end{align*}

  \subsection{Case $\g_\R=\mathfrak{e}_{6(-14)}  $}From \cite[Table 1]{BKOP} and \cite{EHW}, we have $$w_c=w_{D_{5}}=
      s_2s_{3}(s_{4}s_2s_{3}s_{4})(s_5 s_{4}s_2s_{3}s_{3} s_5)(s_6s_5 s_{4}s_2 s_{3}s_4s_5 s_{6}).
$$
From Bourbaki \cite[PLATE IV]{Bour}, we know that $w_c(x_1,\dots,x_8)=(x_1,-x_2,-x_3,-x_4,-x_5,x_6,x_7,x_8)$ for any vector $(x_1,\dots,x_8)\in \mathbb{R}^8$.

For  $1\leq k \leq r=2$,    we  have $h^{\vee}-kc=12-3k$ by Table \ref{constants-k}.   Note that  $\Gamma'\simeq \mathfrak{su}(1, p)$ with $1\leq p\leq 5$ or $\mathfrak{so}(2,8)$. Then $h'^{\vee}=p+1$ or $8$.    

When $k=1$ and $\Gamma'\simeq \mathfrak{su}(1, p)$ with $1\leq p\leq 5$,  from Corollary \ref{numberNk} we will have $p+1=12-3k$ since $h'^{\vee}=h^{\vee}-kc$, which simplifies to $p=11-3=8$. This is a contradiction since $p\leq 5$. When $\Gamma'\simeq \mathfrak{so}(2, 8)$, from Corollary \ref{numberNk} we will have $8=12-3k$ since $h'^{\vee}=h^{\vee}-kc$, which simplifies to $8=9$. This is a contradiction. Thus $N_1=0$.

Note that there are two different subdiagrams which are isomorphic to $\mathfrak{su}(1, 4)$. We denote them by $\mathfrak{su}(1, 4)^I$ (with $S(\Gamma')=\{1,3,4,5\}$) and $\mathfrak{su}(1, 4)^{II}$ (with $S(\Gamma')=\{1,3,4,2\}$).

When $k=r=2$ and $\Gamma'\simeq\mathfrak{su}(1, p)$ with $1\leq p\leq 5$,  from Corollary \ref{numberNk} we will have $p+1<12-3=9$ since $h'^{\vee}<h^{\vee}-(r-1)c$, which simplifies to $p\leq 7$.  When $\Gamma'\simeq \mathfrak{so}(2, 8)$, from Corollary \ref{numberNk} we will have $8<12-3=9$ since $h'^{\vee}<h^{\vee}-(r-1)c$, which simplifies to $8<9$. Thus $N_2=6+1+1=8$.

Therefore, we have 
\begin{align*}N_k&=\begin{cases} 
0, &\text{if~} k=1  \\
8, &\text{if~} k=2.
         \end{cases}
\end{align*}

From \cite[\S 3.7]{EHP}, the  generalized Young diagrams of the corresponding $x\in {}^\mathfrak{k}W$ for $\Gamma'$ are as follows in Table \ref{e6-diagram}:

\begin{center}
    \begin{table}[h]
        
        {
            \renewcommand{\arraystretch}{1.0} % 增加行高
            \setlength{\extrarowheight}{3pt} % 增加额外的行高
            \begin{tabular}{c|c|c}
                \hline
                $S(\Gamma')$ & Generalized Young diagram & $\lambda_{\Phi'}$ \\
                \hline
                $\varnothing$ & $\varnothing$ & $-12\omega_1$ \\
                \hline
                $\{1\}$ & \raisebox{-0.5ex}{\ytableausetup{smalltableaux}
                    \begin{ytableau}
                        1
                    \end{ytableau}} & $-12\omega_1+\omega_2$ \\
                \hline
                $\{1,3\}$ & \raisebox{-0.5ex}{\ytableausetup{smalltableaux}
                    \begin{ytableau}
                        1 & 3
                    \end{ytableau}} & $-12\omega_1+\omega_4$ \\
                \hline
                $\{1,3,4\}$ & \raisebox{-0.5ex}{\ytableausetup{smalltableaux}
                    \begin{ytableau}
                        1 & 3 & 4
                    \end{ytableau}} & $-12\omega_1+\omega_3+\omega_5$ \\
                \hline
                $\{1,3,4,5\}$ & \raisebox{-0.5ex}{\ytableausetup{smalltableaux}
                    \begin{ytableau}
                        1 & 3 & 4 & 5
                    \end{ytableau}} & $-12\omega_1+2\omega_3+\omega_6$ \\
                \hline
                $\{1,3,4,2\}$ & \raisebox{-0.5ex}{\ytableausetup{smalltableaux}
                    \begin{ytableau}
                        1 & 3 & 4 \\
                        \none & \none & 2
                    \end{ytableau}} & $-11\omega_1+2\omega_5$ \\
                \hline
                $\{1,3,4,5,2\}$ & \raisebox{-0.1ex}{\ytableausetup{smalltableaux}
                    \begin{ytableau}
                        1 & 3 & 4 & 5 \\
                        \none & \none & 2 & 4 \\
                        \none & \none & \none & 3 \\
                        \none & \none & \none & 1
                    \end{ytableau}} & $-8\omega_1+4\omega_6$ \\
                \hline
                $\{1,3,4,5,6\}$ & \raisebox{-0.5ex}{\ytableausetup{smalltableaux}
                    \begin{ytableau}
                        1 & 3 & 4 & 5 & 6
                    \end{ytableau}} & $-12\omega_1+3\omega_3$ \\
                \hline
                $\{1,2,3,4,5,6\}$ & \raisebox{-0.5ex}{\ytableausetup{smalltableaux}
                    \begin{ytableau}
                        1 & 3 & 4 & 5 & 6 \\
                        \none & \none & 2 & 4 & 5 \\
                        \none & \none & \none & 3 & 4 & 2 \\
                        \none & \none & \none & 1 & 3 & 4 & 5 & 6
                    \end{ytableau}} & $0$ \\
                \hline
            \end{tabular}
        }\caption{Generalized Young diagrams of $\mathfrak{e}_{6(-14)}$}\label{e6-diagram}
    \end{table}
\end{center}
The corresponding $x$ for each subdiagram $\Gamma'$ can be obtained analogously to the previous subsections.

 \subsection{Case $\g_\R=\mathfrak{e}_{7(-25)}  $}From \cite[Table 1]{BKOP} and \cite{EHW}, we have $$w_c=w_{E_{6}}=
      s_5s_{4}(s_{3}s_5s_{4}s_{3})(s_2 s_{3}s_5s_{4}s_{3} s_2)(s_1s_2 s_{3}s_5 s_{4}s_3s_2 s_{1})(s_6s_5s_3s_4s_2s_1s_3s_2s_5s_3s_4s_6s_5s_3s_2s_1).
$$
From Bourbaki \cite[PLATE IV]{Bour}, we know that \begin{align*}
    w_c(x_1,\dots,x_8)=&\frac{1}{2}(-x_1-x_2-x_3+2x_4,-3x_1-3x_2-x_3+2x_4,-x_1+x_2-x_3-x_4,\\
    &x_1-x_2-x_3-x_4,x_5+3x_6,x_5-x_6,x_5-x_6,-x_5+x_6)
\end{align*} for any vector $(x_1,\dots,x_8)\in \mathbb{R}^8$.

For  $1\leq k \leq r=3$,    we  have $h^{\vee}-kc=18-4k$ by Table \ref{constants-k}.   Note that  $\Gamma'\simeq\mathfrak{su}(p, 1)$ with $1\leq p\leq 6$ or $\mathfrak{so}(2,10)$. Then $h'^{\vee}=p+1$ or $10$.

When $k=1$ and $\Gamma'\simeq \mathfrak{su}(p, 1)$ with $1\leq p\leq 6$,  from Corollary \ref{numberNk} we will have $p+1=18-4k$ since $h'^{\vee}=h^{\vee}-kc$, which simplifies to $p=13$. This is a contradiction since $p\leq 6$. When $\Gamma'\simeq \mathfrak{so}(2, 10)$, from Corollary \ref{numberNk} we will have $10=18-4$ since $h'^{\vee}=h^{\vee}-kc$, which simplifies to $10=14$. This is a contradiction. Thus $N_1=0$.

Note that there are two different subdiagrams which are isomorphic to $\mathfrak{su}(5, 1)$. We denote them by $\mathfrak{su}(5, 1)^I$ (with $S(\Gamma')=\{3,4,5,6,7\}$) and $\mathfrak{su}(5, 1)^{II}$ (with $S(\Gamma')=\{2,4,5,6,7\}$).

When $k=2$ and $\Gamma'\simeq\mathfrak{su}(p, 1)$ with $1\leq p\leq 6$,  from Corollary \ref{numberNk} we will have $p+1=18-4k$ since $h'^{\vee}=h^{\vee}-kc$, which simplifies to $p=9$. This is a contradiction.  When $\Gamma'\simeq \mathfrak{so}(2, 10)$, from Corollary \ref{numberNk} we will have $10=18-4k$ since $h'^{\vee}=h^{\vee}-kc$, which simplifies to $10=10$. Thus $N_2=1$.
  
When $k=r=3$ and $\Gamma'\simeq\mathfrak{su}(1, p)$ with $1\leq p\leq 6$,  from Corollary \ref{numberNk} we will have $p+1<18-8$ since $h'^{\vee}<h^{\vee}-(r-1)c$, which simplifies to $p\leq 8$.  When $\Gamma'\simeq \mathfrak{so}(2, 10)$, from Corollary \ref{numberNk} we will have $10<18-8$ since $h'^{\vee}<h^{\vee}-(r-1)c$, which simplifies to $10<10$. This is a contradiction. Thus $N_2=7+1=8$.

Therefore, we have 
\begin{align*}N_k&=\begin{cases} 
0, &\text{if~} k=1  \\
1, &\text{if~} k=2 \\
8, &\text{if~} k=3.
         \end{cases}
\end{align*}

From \cite[\S 3.7]{EHP}, the  generalized Young diagrams of the corresponding $x\in {}^\mathfrak{k}W$ for $\Gamma'$ are as follows in Table \ref{e7-diagram}:

\begin{center}
    \begin{table}[H]
        
        {
            \renewcommand{\arraystretch}{1} % 增加行高
            \setlength{\extrarowheight}{3pt} % 增加额外的行高
            \begin{tabular}{c|c|c}
                \hline
                $S(\Gamma')$ & Generalized Young diagram & $\lambda_{\Phi'}$ \\
                \hline
                $\varnothing$ & $\varnothing$ & $-18\omega_7$ \\
                \hline
                $\{7\}$ & \raisebox{-0.5ex}{\ytableausetup{smalltableaux}
                    \begin{ytableau}
                        7
                    \end{ytableau}} & $\omega_1-18\omega_7$ \\
                \hline
                $\{7,6\}$ & \raisebox{-0.5ex}{\ytableausetup{smalltableaux}
                    \begin{ytableau}
                        7 & 6
                    \end{ytableau}} & $\omega_3-18\omega_7$ \\
                \hline
                $\{7,6,5\}$ & \raisebox{-0.5ex}{\ytableausetup{smalltableaux}
                    \begin{ytableau}
                        7 & 6 & 5
                    \end{ytableau}} & $\omega_4-18\omega_7$ \\
                \hline
                $\{7,6,5,4\}$ & \raisebox{-0.5ex}{\ytableausetup{smalltableaux}
                    \begin{ytableau}
                        7 & 6 & 5 & 4
                    \end{ytableau}} & $\omega_2+\omega_5-18\omega_7$ \\
                \hline
                $\{7,6,5,4,3\}$ & \raisebox{-0.5ex}{\ytableausetup{smalltableaux}
                    \begin{ytableau}
                        7 & 6 & 5 & 4 & 3
                    \end{ytableau}} & $2\omega_2+\omega_6-18\omega_7$ \\
                \hline
                $\{7,6,5,4,3,1\}$ & \raisebox{-0.5ex}{\ytableausetup{smalltableaux}
                    \begin{ytableau}
                        7 & 6 & 5 & 4 & 3 & 1
                    \end{ytableau}} & $3\omega_2-17\omega_7$ \\
                \hline
                $\{7,6,5,4,2\}$ & \raisebox{-0.5ex}{\ytableausetup{smalltableaux}
                    \begin{ytableau}
                        7 & 6 & 5 & 4 \\
                        \none & \none & \none & 2
                    \end{ytableau}} & $2\omega_5-18\omega_7$ \\
                \hline
                $\{7,6,5,4,3,2\}$ & \raisebox{-0.5ex}{\ytableausetup{smalltableaux}
                    \begin{ytableau}
                        7 & 6 & 5 & 4 & 3 \\
                        \none & \none & \none & 2 & 4 \\
                        \none & \none & \none & \none & 5 \\
                        \none & \none & \none & \none & 6 \\
                        \none & \none & \none & \none & 7
                    \end{ytableau}} & $5\omega_6-18\omega_7$ \\
                \hline
                $\{7,6,5,4,3,2,1\}$ & \raisebox{-0.5ex}{\ytableausetup{smalltableaux}
                    \begin{ytableau}
                        7 & 6 & 5 & 4 & 3 & 1 \\
                        \none & \none & \none & 2 & *(white) 4 & *(white) 3 \\
                        \none & \none & \none & \none & *(white) 5 & *(white) 4 & *(white) 2 \\
                        \none & \none & \none & \none & *(white) 6 & *(white) 5 & *(white) 4 & *(white) 3 & *(white) 1 \\
                        \none & \none & \none & \none & *(white) 7 & *(white) 6 & *(white) 5 & *(white) 4 & *(white) 3 \\
                        \none & \none & \none & \none & \none & \none & \none & *(white) 2 & *(white) 4 \\
                        \none & \none & \none & \none & \none & \none & \none & \none & *(white) 5 \\
                        \none & \none & \none & \none & \none & \none & \none & \none & *(white) 6 \\
                        \none & \none & \none & \none & \none & \none & \none & \none & *(white) 7
                    \end{ytableau}} & $0$ \\
                \hline
            \end{tabular}
        }\caption{Generalized Young diagrams of $\mathfrak{e}_{7(-25)}$}\label{e7-diagram}
    \end{table}
\end{center}

The corresponding $x$ for each subdiagram $\Gamma'$ can be obtained analogously to the previous subsections.

    \subsection*{Acknowledgments}
	
	Z. Bai was supported  by the National Natural Science Foundation of 
	China (No. 12171344).
 
 %\newpage
 % and the National Key $\textrm{R}\,\&\,\textrm{D}$ Program of China (No. 2018YFA0701700 and No. 2018YFA0701701).

\bibliographystyle{plain}
\bibliography{BEHJ-2025}

\end{document}